\documentclass[a4paper,12pt,reqno]{article}
\usepackage[margin=3cm,footskip=1cm]{geometry}
\usepackage[utf8]{inputenc}
\usepackage[T1]{fontenc}
\usepackage[shortlabels]{enumitem}
\usepackage{hyperref}
\usepackage{url}
\usepackage[round]{natbib}

\usepackage{etoolbox}
\AtBeginEnvironment{thebibliography}{%
  \small
  \setlength\itemsep{0.75em plus 0.5em minus 0.75em}%
}

\usepackage{fixltx2e}
\usepackage{soul}
\usepackage{amsmath,amssymb,amsthm,bm}
\allowdisplaybreaks
\usepackage{calc,mathtools,mathstyle}
\mathtoolsset{centercolon}
\numberwithin{equation}{section}

\usepackage{bbm}

\usepackage[sort,capitalise,noabbrev,nameinlink]{cleveref}
\newcommand{\crefrangeconjunction}{--}
\crefname{assumption}{Assumption}{Assumptions}

\usepackage{setspace,xspace}
\makeatletter
\renewenvironment{abstract}{%
  \small%
  \providecommand\keywords{%
    \par\medskip\noindent\textit{Keywords:}\xspace}%
  \providecommand\amssc{%
    \par\medskip\noindent\textit{AMS Subject Classification:}\xspace}%
  \begin{center}%
    \bfseries \abstractname\vspace{-.5em}\vspace{\z@}%
  \end{center}%
  \quote%
}{\endquote}
\makeatother

\makeatletter
\newlength\dlf
\newsavebox\hest
\newsavebox\hestii
\def\ddanglefactor{0.5}

\newcommand\dda@measure[1]{
  \global\sbox\hestii{$\m@th\currentmathstyle
    \left\langle\vphantom{#1}\right.$}
}

\newcommand\dda@autoscale[1]{
  \dda@measure{#1}
  \left\langle\kern-\ddanglefactor\wd\hestii\left\langle
      #1
    \right\rangle\kern-\ddanglefactor\wd\hestii\right\rangle
}

\newcommand\dda@manualscale[2][\@gobble]{
  \@nameuse{\expandafter\@gobble\string #1l}\langle
  \kern-\ddanglefactor\wd\hestii
  \@nameuse{\expandafter\@gobble\string #1l}\langle
  #2
  \@nameuse{\expandafter\@gobble\string #1r}\rangle
  \kern-\ddanglefactor\wd\hestii
  \@nameuse{\expandafter\@gobble\string #1r}\rangle
}

\newcommand\ddangle{\@ifstar{\dda@autoscale}{\dda@manualscale}}

\makeatother

\newcommand{\sB}{\mathcal{B}}
\newcommand{\sD}{\mathcal{D}}
\newcommand{\sF}{\mathcal{F}}
\newcommand{\sG}{\mathcal{G}}
\newcommand{\sH}{\mathcal{H}}
\newcommand{\sK}{\mathcal{K}}
\newcommand{\sS}{\mathcal{S}}
\newcommand{\sT}{\mathcal{T}}

\newcommand{\bN}{\mathbbm{N}}
\newcommand{\bR}{\mathbbm{R}}

\DeclarePairedDelimiter\norm{\lVert}{\rVert}
\DeclarePairedDelimiter\abs{\lvert}{\rvert}

\let\phi\varphi

\DeclareMathOperator{\leb}{Leb}
\DeclareMathOperator{\supp}{supp}
\DeclareMathOperator{\dom}{Dom}

\newcommand{\eqdef}{\coloneqq}

\newcommand{\vmbv}{\ensuremath{\mathcal{VMBV}}}
\newcommand{\vmlv}{\ensuremath{\mathcal{VMLV}}}
\newcommand{\lss}{\ensuremath{\mathcal{LSS}}}

\theoremstyle{plain}
\newtheorem{theorem}{Theorem}[section]

\newtheorem{proposition}[theorem]{Proposition}
\newtheorem{corollary}[theorem]{Corollary}

\newtheorem{fact}[theorem]{Fact}
\newtheorem{assumption}{Assumption}
\theoremstyle{definition}
\newtheorem{definition}[theorem]{Definition}
\newtheorem{remark}[theorem]{Remark}

\usepackage{authblk}

\newcommand{\email}[1]{\href{mailto:#1}{\protect\nolinkurl{#1}}}

\begin{document}
\title{On stochastic integration for volatility modulated
Brownian-driven Volterra processes via white noise analysis}

\author[1]{Ole E. Barndorff-Nielsen}
\author[2]{Fred Espen Benth}
\author[3]{Benedykt Szozda}

\affil[1]{Thiele Centre, Department of Mathematics, \&
  \mbox{CREATES}, School of Economics and Management, Aarhus
  University, Ny Munkegade 118, DK-8000 Aarhus C, Denmark,
  \email{oebn@imf.au.dk}}

\affil[2]{Centre of Mathematics for Applications, University of Oslo,
  P.O. Box 1053, Blindern, N-0361 Oslo, Norway,
  \email{fredb@math.uio.no}}

\affil[3]{Thiele Centre, Department of Mathematics, \&
  \mbox{CREATES}, School of Economics and Management, Aarhus
  University, Ny Munkegade 118, DK-8000 Aarhus C, Denmark,
  \email{szozda@imf.au.dk}}

\date{}

\maketitle

\begin{abstract}
  This paper generalizes the integration theory for volatility
  modulated Brow\-nian-driven Volterra processes onto the space \(
  \sG^{*} \) of Potthoff--Timpel distributions. Sufficient conditions
  for integrability of generalized processes are given, regularity
  results and properties of the integral are discussed. We introduce a
  new volatility modulation method through the Wick product and
  discuss its relation to the pointwise-multiplied volatility model.


  \amssc 60H05, 60H40, 60H07, 60G22
\end{abstract}


\section{Introduction}
\label{sec:intro}
Recently, \citet{BNBPV:2012} developed a theory of stochastic
integration with respect to volatility modulated L\'evy-driven
Volterra processes (\vmlv{}), that are  stochastic
integrals of the form
\begin{equation}
\label{eq:23}
  \int_0^tY(t)\,dX(t),\qquad\text{where}\qquad X(t)=\int_0^t g(t,s)\sigma(s)\,dL(s).
\end{equation}
Here, \( g \) is a deterministic kernel, \( \sigma \) is a stochastic
process embodying the volatility and \( L(t) \) is a L\'evy process.  When \( L(t)=B(t) \) is
the standard Brownian motion, the process \( X(t) \) is termed a
volatility modulated Brownian-driven Volterra process (\vmbv), and this is
the class of processes that we will concentrate our attention on in
this paper; so from now on we fix \( L=B \) in \cref{eq:23}.

\cite{BNBPV:2012} use methods of Malliavin calculus to
validate the following definition of the integral:
\begin{equation}
  \label{eq:10}
  \int_0^t Y(s)\,dX(s)
  = \int_0^t \sK_g(Y)(t,s)\sigma(s)\,\delta^M B(s)
  + \int_0^tD_s^M \left( \sK_g(Y)(t,s) \right)\sigma(s)\,ds,
\end{equation}
where
\begin{equation*}
  \sK_g(Y)(t,s) = Y(s)g(t,s) + \int_s^t ( Y(u)-Y(s) )\,g(du,s),
\end{equation*}
\( \delta^M B(s) \) denotes the Skorohod integral and \( D_t^M \) is
the Malliavin derivative. The superscript \( M \) is used above to
stress that the operators are defined in the Malliavin calculus
setting, but as we will show, the only difference between these
operators and the ones used in the forthcoming sections is the
restriction of the domain. The only results needed to establish the
above definition are the Malliavin calculus versions of the
``fundamental theorem of calculus'' and the ``integration by parts
formula.''

Before we begin the theoretical discussion, let us review some of the
literature that is closely related to the problems addressed in this
paper. The results presented in the following sections are extending
the results from the already mentioned work of \citet{BNBPV:2012} and
those results are in turn generalizing (among others) the results of
\citet{AMN:2001,D:2002,D:2005}. Note that the operator \( \sK_g(\cdot)
\) used by \citet{BNBPV:2012} is the same as the operator used by
\citet{AMN:2001}, however the definition of the integral is
different. The latter authors keep only the first integral in the
right-hand side of \cref{eq:10} thus making sure that the expectation
of the integral is zero. The choice between the two definitions should
be based on modelling purposes, but one has to keep in mind that
requiring zero-expectation in the non-semimartingale setting might be
unreasonable.

It should be noted, that the \vmlv{} processes are a superclass of the
L\'evy semistationary processes (\lss) and a subclass of ambit
processes (more precisely, null-spatial ambit processes). In order to
obtain an \lss{} process from the general form of \vmlv{} process, we
take \( g \) to be a \emph{shift-kernel}, that is \( g(t,s)=g(t-s)
\). Examples of such kernels include the Ornstein--Uhlenbeck kernel
(\( g(u)=e^{-\alpha u},\alpha>0 \)) and a function often used in
turbulence (\( g(u)=u^{\nu-1}e^{-\alpha u},\alpha>0,\nu>1/2 \)). On
the other hand, if \( g(t,s) = c(H)(t-s)^{H-1/2}+c(H)\left(
  \tfrac{1}{2}-H \right)\int_s^t(u-s)^{H-3/2}\left( 1-(s/u)^{1/2-H}
\right)\,du \), where \( c(H)=\left( 2H\Gamma(3/2-H)
\right)^{1/2}\left( \Gamma(H+1/2)\Gamma(2-2H) \right)^{-1/2} \), with
\( H\in(0,1) \) then
\begin{equation*}
X(t)=\int_0^tg(t,s)\,dB(s)
\end{equation*}
is the fractional Brownian motion with Hurst parameter \( H \).

As pointed out by \citet{BNBPV:2012} and illustrated above, the class
of \vmlv{} processes is very flexible as it has already been
applied in modelling of a wide range of naturally occurring random
phenomena. \vmlv{} processes have been studied in the context of
financial data \citep{BNBV:2010,BNBV:2011,VV:2013} and in connection
with turbulence \citep{BNS:2008,BNS:2009}.

As mentioned in \citet{BNBPV:2012}, there are several properties of the integral defined in \cref{eq:10} that one might find important in applications.
Firstly, the definition of the integral does not require adaptedness of the integrand.
Secondly, the kernel function \( g(t,s) \) can have a singularity at \( t=s \) (for example the shift-kernel used in turbulence and presented above.)
Finally, the integral allows for integration with respect to non-semimartingales (as illustrated above by the fractional Brownian motion.)

Our approach allows to treat less regular stochastic processes than
the approach of \cite{BNBPV:2012} because we are not limited to a
subspace of square-integrable random variables. The price we have to
pay with the white noise approach is that the integral might not be a
square-integrable random variable. However, the choice of the \(
\sG^{*} \) space as the domain of consideration has its advantages, as
we can approximate any random variable from \( \sG^{*} \) by
square-integrable random variables. We discuss the properties of the
spaces we work on in the forthcoming sections.

We consider the definition of the integral in \cref{eq:10} in the
white noise analysis setting. We concentrate mostly on the so-called
Potthoff--Timpel space \( \sG^{*} \) and it is important to note here
that this space is much larger than the space of square-integrable
random variables and thus our results extend those of
\citet{BNBPV:2012} considerably.  We review the relevant parts of
white noise analysis in \cref{sec:WNA-background}. In
\cref{sec:calculus} we show that the Malliavin derivative \( D_t^M \)
can be generalized to an operator \( D_t\colon\sG^{*}\to\sG^{*} \) as
can the Skorohod integral. Moreover, we obtain a version of the
``fundamental theorem of calculus'' and ``integration by parts
formula'' in the new setting, making it possible to retrace the steps
taken by \cite{BNBPV:2012} in the heuristic derivation of the
definition of the \vmbv{} integral.

In \cref{sec:integration-vmbv} we first examine regularity of the
operator \( \sK_g(\cdot) \) in the white noise setting. Next, we
consider the case without volatility modulation, that is \(
\sigma=1\). In \cref{sec:integration-volatility} we introduce the
volatility modulation in two different situations. Namely, we consider
\( \sigma \) to be a test stochastic process that multiplies the
kernel function \( g \). We also study the \vmlv{} processes in which
volatility modulation is introduced through the Wick product. This
allows us to consider generalized stochastic processes as the
volatility. In the case that the volatility is a generalized process
that is strongly independent of \( \sK_g(Y) \), we show the
equivalence of the definition of the integrals using the Wick and
pointwise products. In all three cases, we establish mild conditions
on the integrand that ensure the existence of the integral and obtain
regularity results. In \cref{sec:prop-integral-vmbv} we explore the
properties of the integral and in \cref{sec:example} we give an
example which cannot be treated with the methods of
\citet{BNBPV:2012}.

\section{A brief background on white noise analysis}
\label{sec:WNA-background}
In this section we present a brief background on Gaussian white noise
analysis. We will discuss only the relevant parts of this vast theory,
and refer an interested reader to standard books
\citet{HKPS:1993,HOUZ:2010,K:1996} and references therein for a more
complete discussion of this topic.

In order to simplify the exposition of what follows, we recall some
standard notation that will be used throughout this paper.  We denote
by \( (\cdot,\cdot)_{\sH} \) and \( \abs*{\cdot}_{\sH} \) an inner
product and a norm of a Hilbert space \( \sH \), and by \(
\widehat{\cdot} \) the symmetrization of functions or function spaces.

Let \(\sS(\bR)\) denote the Schwartz space of rapidly decreasing
smooth functions and \(\sS'(\bR)\) be its dual, that is the space of
tempered distributions, and let \( \langle\cdot,\cdot\rangle \) denote
the bilinear pairing between \( \sS'(\bR) \) and \( \sS(\bR) \).  By
the Bochner--Minlos theorem, there exists a Gaussian measure \(\mu\)
on \(\sS'(\bR)\) defined through
\begin{equation*}
  \int_{\sS'(\bR)}e^{i\langle x,\xi\rangle}\,d\mu(x)
  =e^{-\frac{1}{2}|\xi|_{L^2(\bR)}^{2}},\qquad \xi\in\sS(\bR).
\end{equation*}
From now on, we take \((\Omega,\sF,P)\eqdef
(\sS'(\bR),\sB(\sS'(\bR)),\mu)\) as the underlying probability space,
where \(\sB(\sS'(\bR))\) is the Borel \(\sigma\)-field of subsets of
\(\sS'(\bR)\).

Observe that we can reconstruct the spaces \(\sS(\bR)\) and
\(\sS'(\bR)\) as nuclear spaces.  We recall this construction briefly,
as a similar one will be used in the definition of spaces of test and
generalized random variables \( \sG,\sG^{*},(\sS) \) and \( (\sS)^{*}
\).  Start with a family of seminorms \(\abs*{\cdot}_p\), with \(
p\in\bR \), defined by
\begin{equation*}
  \abs*{f}_p\eqdef\abs*{(A)^pf}_{L^2(\bR)},\qquad f\in L^2(\bR),
\end{equation*}
where \(A=-\tfrac{d^2}{dx^2}+(1+x^2)\) is a second order differential
operator densely defined on \(L^2(\bR)\).  We denote by \(\sS_p(\bR)\)
the space of those \(f\in L^{2}(\bR)\) that have finite
\(\abs*{\cdot}_p\) norm.  The Schwartz space of rapidly decreasing
functions is the projective limit of spaces \(\{
  \sS_{p}(\bR)\colon p>0 \}\) and the space of tempered
distributions is its dual, or the inductive limit of spaces \(\{
  \sS_{-p}(\bR)\colon p>0 \}\). Note that we have the inclusions
\(\sS(\bR)\subset L^2(\bR)\subset \sS'(\bR)\).

Let \((L^2)=L^2(\sS'(\bR),\mu)\).  By the Wiener--It\^o decomposition
theorem, for any \(\phi\in(L^2)\) there exists a unique sequence of
symmetric functions \(\phi^{(n)}\in\widehat{L^2}(\bR^n)\) such that
\begin{equation}
  \label{eq:1}
  \phi=\sum_{n=0}^\infty I_n(\phi^{(n)}),
\end{equation}
where \(I_n\) is the \(n\) times iterated Wiener integral.  Moreover,
the \((L^2)\) norm of \(\phi\) can be expressed as
\begin{equation*}
  \norm*{\phi}_{(L^2)}^2=\sum_{n=0}^\infty n!\abs*{\phi^{(n)}}_{L^2(\bR^n)}^2.
\end{equation*}
Let us remark, that we will keep the convention of naming the kernel
functions of the chaos expansion of \( \phi \) by \( \phi^{(n)} \).

Next, we recall two types of spaces of test and generalized random
variables. The construction of these spaces follows the construction
of the Schwartz spaces of test and generalized functions. The first
pair we discuss below are the Hida spaces.

For any \(\phi\in(L^2)\) and \(p\in\bR\) define the following norm
\begin{equation*}
  \norm*{\phi}_p^2
  \eqdef
  \sum_{n=0}^{\infty}n!\abs*{(A^{\otimes n})^p\phi^{(n)}}_{L^2(\bR^n)}^2
\end{equation*}
and a corresponding space
\begin{equation*}
  (\sS)_p\eqdef \{ \phi\in(L^2)\colon \norm*{\phi}_p<\infty \}.
\end{equation*}
It is easy to show that for \(p>q\) the following inclusion holds
\((\sS)_p\subset(\sS)_q\).  We define the Hida space of test functions
\((\sS)\) as the projective limit of \(\{(\sS)_p\colon p>0\}\) and the
Hida space of generalized functions as its dual \((\sS)^{*}\).  Note
that \((\sS)^{*}\) can also be defined as the inductive limit of the
spaces \(\{(\sS)_{-p}\colon p>0\}\).  The bilinear pairing between
spaces \( (\sS)^{*} \) and \( (\sS) \) is denoted by \(
\ddangle*{\cdot,\cdot} \) and we have
\begin{equation*}
  \ddangle*{\Phi,\phi}
  \eqdef
  \sum_{n=0}^{\infty}n!\langle\Phi^{(n)},\phi^{(n)}\rangle.
\end{equation*}

The second pair of function spaces are the spaces that were studied
(among others) by \citet{PT:1995} and are denoted by \(\sG\) and
\(\sG^{*}\).  These spaces are constructed through \((L^2)\) norms with
exponential weights of the number operator (sometimes also called
Ornstein--Uhlenbeck operator).  The number operator \(N\) can be
defined through its action on the chaos expansion.  It multiplies the
\(n\)-th chaos by \(n\), that is if
\(\phi=\sum_{n=0}^{\infty}I_n(\phi^{(n)}) \), then
\(N\phi=\sum_{n=0}^{\infty}nI_n(\phi^{(n)})\).

For any \(\lambda\in\bR\) define the norm
\begin{equation*}
  \norm*{\phi}_{\lambda}^2
  \eqdef
  \norm*{e^{\lambda N}\phi}_{(L^2)}^2
  = \sum_{n=0}^{\infty}n!e^{2\lambda n}\abs*{\phi^{(n)}}_{L^2(\bR^n)}^2
\end{equation*}
and a corresponding space
\begin{equation*}
  \sG_{\lambda}
  \eqdef
  \{ \phi\in(L^2) \colon \norm*{\phi}_{\lambda}<\infty  \}.
\end{equation*}
The space of test random variables \(\sG\) is the projective limit of
spaces \(\{\sG_{\lambda}\colon \lambda>0\}\) and the space of
generalized random variables \(\sG^{*}\) is its dual, or the inductive
limit of \(\{\sG_{-\lambda}\colon \lambda>0\}\). As in the case of the
Hida spaces, we denote the bilinear pairing between \( \sG^{*} \) and
\( \sG \) by \( \ddangle*{\cdot,\cdot} \).

It is a well known fact (see e.g. \citet{K:1996,PT:1995}), that we have
the following proper inclusions
\begin{align*}
  (\sS)\subset\sG
  \subset&{}
  (L^2)\subset\sG^{*}\subset(\sS)^{*},\\
  (\sS)\subset(\sS)_p\subset(\sS)_q
  \subset&{}
  (L^2)\subset(\sS)_{-q}\subset(\sS)_{-p}\subset(\sS)^{*},&&0\leq q\leq p,\\
  \sG\subset\sG_{\lambda}\subset\sG_{\lambda'}
  \subset&{}
  (L^2)\subset\sG_{-\lambda'}\subset\sG_{-\lambda}\subset\sG^{*},&&0\leq \lambda \leq\lambda'.
\end{align*}

Note that, unlike with the space \( (\sS)^{*} \), truncation of an
element of \( \sG^{*} \) is always in \( (L^2) \). This happens
because the kernel functions of \( \sG^{*} \) are elements of \(
L^2(\bR^n) \), and so
\begin{equation}
  \label{eq:15}
  \Phi_N = \sum_{n=0}^N I_n(\Phi^{(n)})\in(L^2)
\end{equation}
because \( \abs*{\Phi^{(n)}}_{L^2(\bR^n)}<\infty \) and a finite sum
of such norms is finite, so \( \norm*{\Phi_N}_{(L^2)}<\infty \).  Thus
we can approximate any \( \sG^{*} \) random variable by \( (L^2) \)
random variables by truncating the chaos expansion as in
\cref{eq:15}. This is not the case with the Hida space \( \sS^{*} \)
because the kernels of Hida random variables are elements of a much
larger Schwartz space \( \sS'(\bR) \) and might have infinite \(
L^2(\bR) \) norms.

\begin{remark}
\label{rem:2}
Note that if \(\phi\in(\sS)\), then \(\norm*{\phi}_p<\infty\) for any
\(p>0\) and if \(\Phi\in\sS^*\), then for some \(q>0\) we have
\(\norm*{\Phi}_{-q}<\infty\).  In this case,
  \begin{equation*}
    \abs*{\ddangle*{ \Phi, \phi}}\leq\norm*{\Phi}_{-q}\norm*{\phi}_{q}.
  \end{equation*}
  Similarly, if \(\phi\in\sG\), then \(\norm*{\phi}_\lambda<\infty\)
  for any \(\lambda>0\) and if \(\Phi\in\sG^*\), then for some
  \(\lambda_0>0\) we have \(\norm*{\Phi}_{-\lambda_0}<\infty\).  And
  again,
  \begin{equation*}
    \abs*{\ddangle*{ \Phi, \phi}}
    \leq
    \norm*{\Phi}_{-\lambda_0}\norm*{\phi}_{\lambda_0}.
  \end{equation*}
\end{remark}

An important tool in white noise analysis is the \(\sS\)-transform
which we define below.
\begin{definition}
  \label{def:3}
  For any \( \Phi\in(\sS)^{*} \) and \( \xi\in\sS(\bR) \), we define
  the \(\sS\)-transform of \( \Phi \) at \( \xi \) as
\begin{equation*}
  \sS (\Phi )(\xi)
  \eqdef
  \ddangle*{ \Phi, e^{\langle\cdot,\xi\rangle-\frac{1}{2}\abs*{\xi}_{L^2(\bR)}^2}}.
\end{equation*}
\end{definition}
Note that \(
e^{\langle\cdot,\xi\rangle-\frac{1}{2}\abs*{\xi}_{L^2(\bR)}^2}\in\sG^{*}
\) for any \( \xi\in\sS(\bR) \), so for any \( \Phi\in\sG^{*} \), the
function \( \sS\Phi \) is everywhere defined on \( \sS(\bR) \)
\cite[see][Example 2.1]{PT:1995}.The importance of the \( \sS
\)-transform is well illustrated by the fact that it is an injective
operator (see \citet{HKPS:1993,K:1996} for details.)  Therefore we
have the following useful result.
\begin{theorem}
  If \( \Phi,\Psi\in(\sS)^{*} \) and \( \sS\Phi=\sS\Psi \) then \(
  \Phi=\Psi \).
\end{theorem}
Thus a generalized function can be uniquely determined by its \( \sS
\)-transform.  Making use of this fact, we can define the Wick product
\( \diamond \) of two distributions.
\begin{definition}
  For \( \Phi,\Psi\in(\sS)^{*} \), we define the Wick product of \(
  \Phi \) and \( \Psi \) as
\begin{equation*}
  \Phi\diamond\Psi\eqdef\sS^{-1}( \sS\Phi\cdot\sS\Psi ).
\end{equation*}
\end{definition}
Alternatively, the Wick product can be expressed in terms of the chaos
expansion by
\begin{equation}
  \label{eq:20}
  \Phi \diamond \Psi
  = \sum_{n,m=0}^{\infty}
  I_{n+m}\left( \Phi^{(n)}\widehat{\otimes}\Psi^{(m)} \right)
  = \sum_{n=0}^{\infty}
  I_n \left(\sum_{m=0}^n \Phi^{(n-m)} \widehat{\otimes}\Psi^{(m)}\right).
\end{equation}
The following is an important fact stating that all of the spaces
considered in this paper, namely \( \sG,\sG^{*},(\sS) \) and \(
(\sS)^{*} \) are closed under the Wick product.
\begin{fact}
  If \( \Phi,\Psi\in\sG \) (or \( \sG^{*},(\sS),(\sS)^{*} \)) then \(
  \Phi\diamond\Psi\in\sG \) (or \( \sG^{*},(\sS),(\sS)^{*} \),
  respectively).
\end{fact}
This is the advantage of using the Wick product instead of the
pointwise product, as the latter is usually not defined on spaces \(
\sG^{*} \) and \( (\sS)^{*} \).  However, under strong independence of
\( \Phi \) and \( \Psi \), the Wick and pointwise products coincide
(see e.g. \citet{BP:1996} for details.)
\begin{definition}
\label{def:2}
  We say that \( \Phi,\Psi\in\sG^{*} \) are strongly independent if
  there are two measurable subsets \( I_{\Phi},I_{\Psi} \) of \( \bR
  \) such that \( \leb(I_{\Phi}\cap I_{\Psi})=0 \) and for all \(
  m,n\in\bN \) we have \( \supp\Phi^{(n)}\subset ( I_{\Phi}
  )^n \) and \( \supp\Psi^{(m)}\subset ( I_{\Psi} )^m
  \).
\end{definition}

From \citet[Proposition 2]{BP:1996} we know that strong independence
and regular independence of random variables are closely
related. Namely, if \( X,Y\in(L^2) \) are two independent random
variables measurable with respect to \( \sigma \{ B(s)\colon a\leq s
<\infty \} \), \( a\in\bR \), then \( Y \) has a version \(
\tilde{Y}\in(L^2) \) such that \( \tilde{Y} \) and \( X \) are
strongly independent.

The next theorem states which products of generalized random variables
are well-defined. The first part (which is a standard result)
deals with the product of generalized and test random variables and the
second part takes advantage of the strong independence assumption. For
the proof of the second part see \cite{BP:1996}.

\begin{theorem}
  \label{thm:6}
  \begin{enumerate}
  \item For \( \Phi\in(\sS)^{*} \) (or \( \sG^{*} \)) and \(
    \phi\in(\sS) \) (or \( \sG \)) the product \( \phi\cdot\Phi \) is
    well-defined through
    \begin{equation*}
      \ddangle*{ \phi\cdot\Phi,\psi}
      = \ddangle*{ \Phi,\psi\cdot\phi},
      \qquad \text{for all }\psi\in(\sS)
      \text{ (or } \sG \text{ respectively)}.
    \end{equation*}
  \item If \( \Phi,\Psi\in\sG^{*} \) are strongly independent, then
    the product \( \Phi\cdot\Psi \) is well-defined, and
    \begin{equation*}
      \Phi\cdot\Psi = \Phi\diamond\Psi.
    \end{equation*}
  \end{enumerate}
\end{theorem}

Next, we state several results that are used to establish some norm
estimates in the following sections of this paper.  First, we recall
an estimate on the norm of a product of two test random variables
given in \citet[Proposition 2.4]{PT:1995}.
\begin{proposition}
  \label{prop:2}
  Let \( \lambda_0\eqdef \tfrac{1}{2}\ln(2+\sqrt{2}) \) and assume
  that \(\lambda>\lambda_0\) and \( \phi,\psi\in\sG_{\lambda}
  \). Then, for all \( \nu>\lambda_0 \), \(
  \phi\cdot\psi\in\sG_{\lambda-\nu} \) and there is a constant \(
  C_{\nu} \) so that
  \begin{equation*}
    \norm*{\phi\cdot\psi}_{\lambda-\nu}
    \leq C_{\nu}\norm*{\phi}_{\lambda}\norm*{\psi}_{\lambda}.
  \end{equation*}
\end{proposition}

Using \cref{prop:2}, we can establish a norm estimate of a pointwise
product of generalized and test random variables.
\begin{theorem}
  \label{thm:5}
  Let \( \lambda_0\eqdef\tfrac{1}{2}\ln(2+\sqrt{2}) \) and assume that
  \(\lambda>\lambda_0\) Suppose that \( \sigma\in\sG \) and \(
  \Phi\in\sG_{-\lambda+\nu}\subset\sG^{*} \), where \( \nu>\lambda
  \). Then there is a constant \( C_{\nu} \) such that
  \begin{equation*}
    \norm*{\sigma\cdot\Phi}_{-\lambda}
    \leq
    C_{\nu}\norm*{\Phi}_{-\lambda+\nu}\norm*{\sigma}_{\lambda}.
  \end{equation*}
\end{theorem}
\begin{proof}
  Consider, for any \( \phi\in\sG \),
  \begin{align*}
    \abs*{\ddangle*{\sigma\cdot\Phi,\phi}} ={}&
    \abs*{\ddangle*{\Phi,\sigma\cdot\phi}}\\
    \leq{}&
    \norm*{\Phi}_{-\lambda+\nu}\norm*{\sigma\cdot\phi}_{\lambda-\nu}\\
    \leq{}& \tilde{C}_{\nu}\norm*{\Phi}_{-\lambda+\nu}
    \norm*{\sigma}_{\lambda}\norm*{\phi}_{\lambda}.
  \end{align*}
  Since the above holds for any \( \phi\in\sG \), there is a constant
  dependent only on \( \nu \) such that
  \begin{equation*}
    \norm*{\sigma\cdot\Phi}_{-\lambda}
    \leq
    C_{\nu}\norm*{\Phi}_{-\lambda+\nu}\norm*{\sigma}_{\lambda}.
  \end{equation*}
  Hence the theorem holds.
\end{proof}

Next, we recall an estimate of the norm of a Wick product of two
generalized random variables from \citet[Proposition 2.6]{PT:1995}.
\begin{proposition}
  \label{prop:3}
  Let \( \Phi,\Psi\in\sG_{\lambda} \), \( \lambda\in\bR \). Let \(
  \lambda_0=\lambda-\tfrac{1}{2} \), and \( \lambda'<\lambda_0
  \). Then \( \Phi\diamond\Psi\in\sG_{\lambda'} \) and
  \begin{equation*}
    \norm*{\Phi\diamond\Psi}_{\lambda'}
    \leq
    C_{\lambda,\lambda'}\norm*{\Phi}_{\lambda}\norm*{\Psi}_{\lambda},
  \end{equation*}
  where \( C_{\lambda,\lambda'} =
  (2(\lambda-\lambda')-1)^{-\tfrac{1}{2}}e^{\lambda-\lambda'-1} \).
\end{proposition}

Finally, let us review the Pettis-type integral in the white noise
setting.  Suppose that \( (\sT,\sB,m) \) is a measure space and \(
\Phi(t)\colon\sT\to(\sS)^{*} \) is a generalized stochastic process.
We say that \( \Phi \) is Pettis-integrable if the following two
conditions are satisfied:
\begin{enumerate}
\item \( \Phi \) is weakly measurable, that is \( t \to \ddangle*{
    \Phi(t), \phi}\) is a measurable function for all \( \phi\in(\sS)
  \);
\item \( \Phi \) is weakly integrable, that is
  \begin{equation*}
    \int_{\sT}\abs*{\ddangle*{ \Phi(t),\phi}}\,dm<\infty,
  \end{equation*}
  for all \( \phi\in(\sS) \).
\end{enumerate}
For a Pettis-integrable generalized process \( \Phi \), we define its
Pettis integral \( \int_{\sT}\Phi(t)\,dm \) by
\begin{equation*}
  \ddangle*{ \int_{\sT}\Phi(t)\,dm,\phi}
  \eqdef \int_{\sT}\ddangle*{\Phi(t),\phi}\,dm.
\end{equation*}
Note that we can derive the chaos expansion of the Pettis white noise
integral (see \citet{HKPS:1993,K:1996} for details), as
\begin{equation*}
  \int_{\sT}\Phi(t)\,dm
  = \sum_{n=0}^\infty I_n \left( \int_{\sT}\Phi^{(n)}(t)\,dm \right),
\end{equation*}
where the integrals in the chaos expansion are understood as Pettis
integrals on the spaces \( \sS'(\bR^n) \) \citep[see][]{P:1938}.
Note that the white noise Pettis integral is defined for processes in
the \( (\sS^{*}) \) space.  However, due to the fact that \(
(\sS)\subset\sG \) and \( \sG^{*}\subset(\sS^{*}) \), we say that a \(
\sG^{*} \)-valued process is Pettis-integrable if it is integrable as
an \( (\sS)^{*} \)-valued process and the result of integration is a
\( \sG^{*} \) random variable.  Alternatively, we can restate the
above definitions requiring that \( \Phi(t)\in\sG^{*} \) and \(
\phi\in\sG \).

In what follows, the fact that Pettis integral and \( \sS \)-transform
are interchangable operations is important.
\begin{proposition}
  \label{prop:7-1}
  For all \( \Phi\in(\sS)^{*} \) and \( \xi\in\sS(\bR) \),
  \[
  \sS \left( \int_0^t\Phi(s)\,ds \right)(\xi) =
  \int_0^t\sS(\Phi(s))(\xi)\,ds.
  \]
\end{proposition}

\section{Calculus in \texorpdfstring{\(\sG^*\)}{G*} and
  \texorpdfstring{\((\sS)^*\)}{(S)*}}
\label{sec:calculus}
\subsection{Stochastic differentiation}
\label{sec:differentiation}
Before we present the definition of the stochastic derivative that we
use in the remainder of this paper, we motivate our choice by showing
how it fits with other definitions that can be found in Malliavin
calculus and white noise analysis.

Let us first recall that the Malliavin derivative is defined on a
subset of \((L^2)\), namely
\begin{equation*}
  \sD_{1,2}
  \eqdef \left\{ \phi\in(L^2) \colon
    \sum_{n=0}^{\infty}n\cdot n!\abs*{\phi^{(n)}}_{L^2(\bR^n)}^2<\infty
  \right\}.
\end{equation*}
For \(\phi\in\sD_{1,2}\) we define the Malliavin derivative by its
chaos expansion as
\begin{equation}
  \label{eq:21}
  D_t^M\phi\eqdef \sum_{n=0}^{\infty}n I_{n-1}( \phi^{(n)}(\cdot,t) ).
\end{equation}
Observe that \( \sD_{1,2} \) is chosen in such a way that \(
D_t^M\phi\in(L^2) \) whenever \( \phi\in\sD_{1,2} \).

In \citet{PT:1995}, the authors define an operator \(D_h\) for any \(h\in
L^2(\bR)\) as the G\^ateaux derivative in direction \(h\).  It can be
shown that \(D_h\) can be described in terms of its chaos expansion as
\begin{equation*}
  D_h\phi
  = \sum_{n=0}^{\infty} nI_{n-1}\left(( h,\phi^{(n)} )_{L^2(\bR)}\right),
\end{equation*}
where \((\cdot,\cdot)_{L^2(\bR)}\) is the \(L^2(\bR)\) inner product,
that is
\begin{equation*}
  ( h,\phi^{(n)} )_{L^2(\bR)}(\underline{u}^{(n-1)})
  \eqdef
  \int_{\bR}h(s)\phi^{(n)}(\underline{u}^{(n-1)},s)\,ds,
  \qquad \underline{u}^{(n-1)}\in\bR^{n-1}.
\end{equation*}
Note that, since \(\phi^{(n)}\) can be assumed to be symmetric, it
does not matter which of the coordinates is chosen as \(s\) in the
formula above.

For \( D_t \) and \( D_h \) to be equal, we need \( h \) to be a
function satisfying
\begin{equation*}
  ( h,\phi^{(n)} )_{L^2(\bR)}
  =\phi^{(n)}(\cdot,t),\qquad\forall\phi^{(n)}\in L^2(\bR^{n}).
\end{equation*}
But there is no \( h\in L^2(\bR) \) that satisfies the above
condition. It is a well-known fact though, that the Dirac delta -- a
generalized function on \( \bR \) -- has this exact property.  We
cannot formally take \(h(s)=\delta_t(s)\), but we can do it informally
to obtain
\begin{align*}
  D_{\delta_t}\phi={}&
  \sum_{n=0}^{\infty} nI_{n-1}\left(( \delta_{t},\phi^{(n)} )_{L^2(\bR)}\right)\\
  ={}&
  \sum_{n=0}^{\infty} nI_{n-1}(\phi^{(n)}(\cdot,t))\\*
  ={}& D_t\phi.
\end{align*}
Note that for the above to hold, we need \(
\phi^{(n)}(\underline{u}^{(n-1)},\cdot) \in \sS(\bR)\) (with \( \underline{u}^{(n-1)}\in\bR^{n-1} \)), as the Dirac
delta is a continuous linear operator on \( \sS(\bR) \). However,
since \( \sS(\bR) \) is a dense subset of \( L^2(\bR) \), the Dirac delta
can be uniquely extended to a densely defined, unbounded linear
functional on \( L^2(\bR) \).  As we will show later, \(
D_{\delta_t}\Phi\in\sG^{*} \) for all \( \Phi\in\sG^{*} \).

In \citet{B:1999}, we encounter yet another differentiation operator.
This time it is defined on the Hida space \((\sS)^{*}\) as \(\sD\Phi =
\Phi\cdot W- \Phi\diamond W\), where (with \(\omega\in\sS'(\bR)\) and
\(f\in\sS(\bR)\)) \(W(f)(\omega)=\langle\omega,f\rangle\) is the
coordinate process sometimes also called a smoothed white noise.  In
this case, the operator \(\sD\) should be understood as a functional
on the product space \(\sS(\bR)\times(\sS)\), with its action given by
\begin{equation*}
  \sD\Phi(f,\phi)
  = (\Phi\cdot W- \Phi\diamond W)(f,\phi)
  = \ddangle*{ \Phi\cdot W(f) - \Phi\diamond W(f),\phi }.
\end{equation*}

In \citet[Proposition 3.3]{B:1999}, it is shown that
operator \(\sD\) can be expressed in terms of the chaos expansion of
the distribution it acts on -- much in the same way as the Malliavin
derivative is defined.  In order to see this, for
\(\Phi^{(n)}\in\widehat\sS'(\bR^n)\),
\(\phi^{(n)}\in\widehat\sS(\bR^n)\) and \(g\in\sS(\bR)\), define
\(\Phi^{(n)}(\cdot,g)\) by
\begin{equation*}
  \langle \Phi^{(n)}(\cdot,g) , \phi^{(n)}
  \rangle \eqdef \langle \Phi^{(n)} , \phi^{(n)} \widehat\otimes g \rangle.
\end{equation*}
Now, the chaos expansion of \( \sD\Phi(g) \) is given by
\begin{equation*}
\sD\Phi(g)=\sum_{n=0}^{\infty} n I_{n-1}(\Phi^{(n)}(\cdot,g)).
\end{equation*}
It is enough to justify that fixing the \(n\)-th functional coordinate
of the functional \(\Phi^{(n)} \colon \sS(\bR) \to \sS'(\bR)\) at a
certain \(g\) is equivalent to fixing the \(n\)-th variable in the
function \(\Phi^{(n)}\).  Suppose that \(\Phi=\sum_{n=0}^\infty
I_n(\Phi^{(n)})\in\sG^*\).  Then, for all \(n\geq0\) the functions
\(\Phi^{(n)}\) are elements of \(L^2(\bR^n)\) and can be viewed as
functions of \(n\) variables or, due to the Riesz representation
theorem, as linear operators acting on \(L^2(\bR^n)\). With
\(\phi^{(n-1)},\Phi^{(n)}\) and \(g\) as above, we have that
\(\phi^{(n-1)}\widehat{\otimes}g\in\widehat{\sS}(\bR^n)\subset
\widehat{L}^2(\bR^n)\), so the bilinear pairing can be viewed as an
inner product in \(L^2(\bR)\).  Thus, with notation
\(\underline{x}^{(n)}=(x_1,x_2,\ldots,x_{n})\),
\(\underline{x}^{(n)}_{\not
  k}=(x_1,x_2,\ldots,x_{k-1},x_{k+1},\ldots,x_{n})\) and
\(d\underline{x}^{(n)}=dx_1dx_2\ldots dx_n \) we have
\begin{align*}
  \left\langle \Phi^{(n)}(\cdot,g),\phi^{(n-1)} \right\rangle ={}&
  \left( \Phi^{(n)}(\cdot,g),\phi^{(n-1)} \right)_{L^2(\bR^n)}\\
  ={}& \frac{1}{n}\sum_{k=1}^n
  \int_{\bR^n}\Phi^{(n)}(\underline{x}^{(n)})
  \phi^{(n-1)}(\underline{x}^{(n-1)}_{\not k})g(x_k)
  \,d\underline{x}^{(n)}.
\end{align*}
Taking, again informally, \(g(x)=\delta_t(x)\) and using the symmetry
of \(\phi^{(n)}\) and \(\Phi^{(n)}\), we have
\begin{align*}
  \left\langle \Phi^{(n)}(\cdot,g),\phi^{(n-1)} \right\rangle ={}&
  \frac{1}{n} \sum_{k=1}^n \int_{\bR^n}
  \Phi^{(n)}(\underline{x}^{(n)})
  \phi^{(n-1)}(\underline{x}^{(n-1)}_{\not k}) \delta_t(x_k)
  \,d\underline{x}^{(n)}\\*
  ={}& \frac{1}{n} \sum_{k=1}^n \int_{\bR^n}
  \Phi^{(n)}(\underline{x}^{(n)}_{\not k},x_{k})
  \phi^{(n-1)}(\underline{x}^{(n-1)}_{\not k})
  \delta_t(x_k)\,dx_k\,d\underline{x}^{(n)}_{\not k}\\
  ={}& \frac{1}{n} \sum_{k=1}^n \int_{\bR^n}
  \Phi^{(n)}(\underline{x}^{(n)}_{\not k},t)
  \phi^{(n-1)}(\underline{x}^{(n-1)}_{\not k})\,d\underline{x}^{(n)}_{\not k}\\
  ={}& \int_{\bR^n}\Phi^{(n)}(\underline{x}^{(n-1)},t)
  \phi^{(n-1)}(\underline{x}^{(n-1)})\,d\underline{x}^{(n-1)}\\
  ={}& \left( \Phi^{(n)}(\cdot,t),\phi^{(n-1)}
  \right)_{L^2(\bR^{n-1})}.
\end{align*}

Thus we have the following informal equality
\begin{equation*}
  D_t\Phi
  = D_{\delta_t}\Phi
  =  \sD \Phi(\delta_t).
\end{equation*}
Therefore, we can regard the derivative defined by \cref{eq:21} as a
restriction of \(\sD\) defined in \citet{B:1999} to the space \(\sG^{*}\),
an extension of the Malliavin derivative \(D_t^M\) onto a larger
domain, and an extension of the derivative \(D_h\) defined in
\citet{PT:1995}. This motivates the following definition.

\begin{definition}
\label{def:derivative}
  For any \(\Phi\in\sG^*\) with chaos expansion given by
  \(\Phi=\sum_{n=0}^\infty I_n(\Phi^{(n)})\) we define the
  \emph{stochastic derivative} of \(\Phi\) at \(t\) by
  \begin{equation*}
    D_t\Phi=\sum_{n=1}^\infty n I_{n-1}(\Phi^{(n)}(\cdot,t)).
  \end{equation*}
\end{definition}

\cref{thm:1} assures that the stochastic derivative is in fact a
well-defined functional acting on \(\sG^{*}\).

\begin{theorem}
  \label{thm:1}
  For any \(\Phi\in\sG^{*}\), we have \(D_t\Phi\in\sG^{*}\) for almost
  all \(t\in\bR\).  Moreover, if for some \(\lambda>0\),
  \(\Phi\in\sG_{-\lambda}\) then for any \(\varepsilon>0\) there is a constant \( C_{\varepsilon} \), such that
  \begin{equation}
    \label{eq:24}
    \int_{\bR}\norm*{D_t\Phi}_{-\lambda-\varepsilon}^2\,dt\leq C_{\varepsilon}\norm*{\Phi}^2_{-\lambda} <\infty,
  \end{equation}
  and in consequence \(D_t\Phi\in\sG_{-\lambda-\varepsilon}\) for
  almost all \( t\in\bR \).
\end{theorem}
\begin{proof}
  It is enough to show that \cref{eq:24} holds because
  \(\sG^{*}=\bigcup_{\lambda>0}\sG_{-\lambda}\).  In order to do this,
  we need the following fact: for any \(\varepsilon>0\), there exists
  \(x_0>e\) such that \(f(x)=\tfrac{\ln x}{x}<\varepsilon\) for all
  \(x>x_0\).  This is a consequence of the fact that \(f(x)\) is
  decreasing on the interval \((e,\infty)\) and
  \(\lim_{x\to\infty}f(x)=0\).

  Let \(\Phi=\sum_{n=0}^{\infty}I_n(\Phi^{(n)}) \) be an element
  of \(\sG_{-\lambda}\), and consider
  \begin{align*}
    \int_{\bR}\norm*{D_t\Phi}_{-\lambda-\varepsilon}^2\,dt ={}&
    \int_{\bR}\sum_{n=0}^{\infty}n(n!)  e^{-2(\lambda + \varepsilon)n}
    \abs*{\Phi^{(n)}(\cdot,t)}_{L^2(\bR^n)}^2\,dt\\
    ={}& \sum_{n=0}^{\infty}n(n!)e^{-2(\lambda + \varepsilon) n}
    \int_{\bR}\abs*{\Phi^{(n)}(\cdot,t)}_{L^2(\bR^n)}^2\,dt\\
    ={}& \sum_{n=0}^{\infty}n(n!)e^{-2(\lambda + \varepsilon)
      n}\abs*{\Phi^{(n)}}_{L^2(\bR^{n+1})}^2.
  \end{align*}
  By the fact stated at the beginning of this proof, we have that for
  any \(\varepsilon>0\) there is a \(k\in\bN_0\) such that for all
  \(n\geq k\) we have \( \tfrac{\ln n}{n}<2\varepsilon\).  This
  ensures that \(n e^{-2(\lambda+\varepsilon)n}\leq e^{-2\lambda n}\).
  Hence
  \begin{align*}
    \sum_{n=k}^{\infty}n(n!)e^{-2(\lambda + \varepsilon)
      n}\abs*{\Phi^{(n)}}_{L^2(\bR^{n+1})}^2 <{}&
    \sum_{n=k}^{\infty}(n!)e^{-2\lambda n}\abs*{\Phi^{(n)}}_{L^2(\bR^{n+1})}^2\\
    \leq{}&
    \norm*{\Phi}^2_{-\lambda}
  \end{align*}
Now, for any \( n\in\{0,1,\ldots,k-1\} \) there is a constant \( c_{n,\varepsilon} \) such that \( ne^{-2(\lambda+\varepsilon)n}<c_{n,\varepsilon}e^{-2\lambda n} \). Let \( \tilde{C}_{\varepsilon}=\max(c_{n,\varepsilon}\colon n\in\{0,1,\ldots,k-1\}) \). We have
\begin{equation*}
  \sum_{n=0}^{k-1}n(n!)e^{-2(\lambda + \varepsilon) n}
  \abs*{\Phi^{(n)}}_{L^2(\bR^{n+1})}^2
  \leq\tilde{C}_{\varepsilon}    \sum_{n=0}^{k-1}(n!)e^{-2(\lambda) n}
  \abs*{\Phi^{(n)}}_{L^2(\bR^{n+1})}^2.
\end{equation*}
  Thus we have shown that
  \begin{equation*}
    \int_{\bR}\norm*{D_t\Phi}_{-\lambda-\varepsilon}\,dt
    \leq\tilde{C}_{\varepsilon}\norm*{\Phi}_{-\lambda}^{2}+\norm*{\Phi}_{-\lambda}^{2}
    \leq C_{\varepsilon}\norm{\Phi}_{-\lambda}^2.
  \end{equation*}
  Therefore \(\norm*{D_t\Phi}_{-\lambda-\varepsilon}<\infty\) for
  almost all \(t\) as required.
\end{proof}

The above theorem improves the result of \citet[Lemma 3.10]{AOPU:2000},
where it was shown that if \(\Phi\in\sG_{-\lambda}\), then
\(D_t\Phi\in\sG_{-\lambda-1}\) for almost all \(t\).  The notation
used in \citet{AOPU:2000} differs from ours, but the definitions of the
spaces \(\sG\) and \(\sG^*\) as well as the definitions of the
stochastic derivative are equivalent.

Recall that \cref{def:derivative} of the stochastic derivative is
exactly the same (in terms of chaos expansion) as the definition of
the Malliavin derivative.  The drawback of the Malliavin derivative is
that it is defined on a smaller space \(\sD_{1,2}\) so that the
derivative takes values in the \((L^2)\) space for almost all \(t\).
Since we define the derivative on a larger space
\(\sG^*\supsetneq\sD_{1,2}\), the result of differentiation also falls
into a larger space, namely \(\sG^*\supsetneq (L^2)\).  Thus the
derivative of a random variable from \(\sG^*\) is no longer an element
of \((L^2)\), but rather a generalized stochastic process.  However,
taking derivative of any test random variable \(\phi\in\sG\) results
in a test stochastic process that is in \(\sG\subsetneq(L^2)\) for
almost all \(t\in\bR\), as can be seen from \cref{thm:1}.

\subsection{Properties of the stochastic derivative}
\label{sec:prop-derivative}
Now we turn our attention to some of the properties of the stochastic
derivative \( D_t \) of \cref{def:derivative}. All of the formulas presented below are well-known in the
setting of Malliavin calculus. We include them for the sake of
completeness and give only sketches of the proofs or omit the proofs
completely.

\begin{proposition}
  If \( \Phi \) is deterministic, that is \( \Phi = I_0(\Phi^{(0)})
  \), \( \Phi^{(0)}\in\bR \), then \( D_t\Phi = 0 \).
\end{proposition}
\begin{proof}
  This is a direct consequence of the definition of the stochastic
  derivative.
\end{proof}

\begin{proposition}
  If \( \Phi,\Psi\in\sG^{*} \), then
  \begin{equation}
    \label{eq:8}
    D_t(\Phi\diamond\Psi)=D_t(\Phi)\diamond\Psi + \Phi\diamond D_t\Psi.
  \end{equation}
\end{proposition}
\begin{proof}
  This follows from straightforward but tedious explicit operations
  on the chaos expansion and comparison of the chaos expansions of the
  left- and right-hand sides of the \cref{eq:8}. The computations are
  the same as in the Malliavin derivative case, as the formulas
  defining the derivatives are the same and only the domain differs.
Existence of both sides of \cref{eq:8} follows from \cref{thm:1,prop:3}
\end{proof}

Similarly, we can show the pointwise product rule with the restriction
that we operate on smooth random variables only.

\begin{proposition}
  If \( \phi,\psi\in\sG \), then
  \begin{equation*}
    D_t(\phi\cdot\psi)=D_t(\phi)\cdot\psi + \phi\cdot D_t\psi.
  \end{equation*}
\end{proposition}

Since pointwise product is not well-defined for random variables in \(
\sG^{*} \), we cannot generalize the above result to all \(
\Phi,\Psi\in\sG^{*} \). However, there are two cases of interest for
which the product rule makes sense. First, under an additional assumption
of strong independence of \( \Phi \) and \( \Psi \), application of
\cref{thm:6} and the fact that the stochastic derivative preserves strong
independence yields:

\begin{proposition}
  If \( \Phi,\Psi\in\sG^{*} \) are strongly independent, then
  \begin{equation*}
    D_t(\Phi\cdot\Psi)=D_t(\Phi)\cdot\Psi + \Phi\cdot D_t\Psi.
  \end{equation*}
\end{proposition}

Finally, since \( \phi\in\sG \) implies that \( D_t\phi\in\sG \) for
almost all \( t \), and the product of test and generalized random
variables is well defined, we obtain:

\begin{proposition}
\label{prop:6}
  If \( \phi\in\sG^{*}\) and \(\Psi\in\sG \), then
  \begin{equation*}
    D_t(\phi\cdot\Psi)=D_t(\phi)\cdot\Psi + \phi\cdot D_t\Psi.
  \end{equation*}
\end{proposition}

Finally, following \citet[Equation (5.17)]{HKPS:1993}, we give the formula for the \( \sS \)-transform of the Malliavin derivative.
In the spirit of completeness, we first recall the definition of the Fr\'echet functional derivative that appears in the formula for the \( \sS \)-transform of the stochastic derivative.
We say that a real-valued function \( f \) defined on an open subset \( U \) of a Banach space \( B \) is Fr\'echet differentiable at \( x \) if there exists a  bounded linear functional \( \tfrac{\delta f}{\delta x}\colon B\to \bR \) such that \( \abs{f(x+y)-f(x)-\tfrac{\delta f}{\delta x}(y)}=o(\norm{y}) \) for all \( y\in B \).
\begin{proposition}
  \label{prop:7-2}
  For all \( \Phi\in(\sS)^{*} \) and \( \xi\in\sS(\bR) \),
   \[
   \sS \left( D_t\Phi \right)(\xi) =
   \tfrac{\delta}{\delta\xi(t)}\sS(\Phi)(\xi),
   \]
   where \( \tfrac{\delta}{\delta\xi(s)} \) is the Fr\'echet
   functional derivative.
\end{proposition}

\subsection{Stochastic integration}
\label{sec:integration}
In this section we introduce the Skorohod integral for processes in
\(\sG^{*}\).  In Malliavin calculus, the Skorohod integral can be
defined through the chaos expansion as
\begin{equation}
  \label{eq:5}
  \phi(t)
  =
  \sum_{n=0}^{\infty} I_n \left( \phi^{(n)}(\cdot,t) \right)
  \Longrightarrow
  \delta^M \left( \phi \right)
  =
  \sum_{n=0}^{\infty} I_{n+1} \left( \widehat{\phi}^{(n)} \right).
\end{equation}
The domain of \(\delta^M\) consists of all those processes whose
Skorohod integral results in a random variable in \( (L^2) \), namely
\begin{equation*}
  \dom \left( \delta^M \right)
  =
  \left\{ \phi\in(L^2) \colon \sum_{n=0}^{\infty} (n+1)!
    \abs*{\widehat{\phi}^{(n)}}_{L^2(\bR^{n+1})}^2 < \infty \right\}.
\end{equation*}
We extend the Skorohod integral in the same manner as we extended the
Malliavin derivative.

\begin{definition}
  \label{def:1}
  For \(\Phi(t)=\sum_{n=0}^{\infty}I_n(\Phi^{(n)}(\cdot,t))
  \in\sG^{*}\), we define the Skorohod integral by
  \begin{equation*}
    \delta(\Phi)
    =
    \int_{\bR}\Phi(t)\,\delta B(t)
    \eqdef
    \sum_{n=0}^{\infty} I_{n+1} \left( \widehat{\Phi}^{(n)} \right),
  \end{equation*}
  whenever \( \sum_{n=0}^{\infty}
  (n+1)!e^{-2(n+1)\lambda}\abs{\widehat{\Phi}^{(n)}}_{L^2(\bR^{n+1}}<\infty
  \) for some \( \lambda>0 \).
\end{definition}

The next result gives sufficient conditions for \(\Phi(t)\) to be
Skorohod-integrable and provides a norm estimate on \( \delta(\Phi) \)
under the assumption of square-integrability of the norm \(
\norm{\delta(\Phi)}_{-\lambda} \).

\begin{theorem}
  \label{thm:2}
  If \(\Phi(t)\in\sG_{-\lambda}\) for all \(t\in\bR\) and
  \begin{equation*}
    \int_{\bR}\norm*{\Phi(t)}_{-\lambda}^2\,dt < \infty,
  \end{equation*}
  then for any \(\varepsilon>0\) there is a constant \( C_{\varepsilon} \) such that
  \begin{equation*}
    \norm*{\delta(\Phi)}_{-\lambda-\varepsilon}^2
    \leq C_{\varepsilon}\int_{\bR}\norm{\Phi(t)}_{-\lambda}^2\,dt.
  \end{equation*}
  Thus \(\delta(\Phi)\in\sG_{-\lambda-\varepsilon}\) and in
  particular, \(\delta(\Phi)\in\sG^{*}\).
\end{theorem}
\begin{proof}
  Fix an arbitrary \(\varepsilon>0\). Keeping in mind that the \(
  L^2(\bR^{n+1}) \) norm of \( \Phi^{(n)}(\cdot,t) \) and its
  symmetrization \( \widehat{\Phi}^{(n)}(\cdot,t) \) are equal,
  consider
  \begin{align}
    \norm*{\delta(\Phi)}_{-\lambda-\varepsilon} ={}&
    \sum_{n=0}^{\infty} (n+1)!e^{-2(\lambda+\varepsilon)n}
    \abs*{\Phi^{(n)}}_{L^2(\bR^{n+1})}^2\nonumber\\
    ={}& \sum_{n=0}^{\infty} (n+1)n!e^{-2(\lambda+\varepsilon)n}
    \int_{\bR} \abs*{\Phi^{(n)}(\cdot,t)}_{L^2(\bR^{n})}^2\,dt\nonumber\\
    ={}& \int_{\bR} \sum_{n=0}^{\infty}
    (n+1)n!e^{-2(\lambda+\varepsilon)n}
    \abs*{\Phi^{(n)}(\cdot,t)}_{L^2(\bR^{n})}^2\,dt.
    \label{eq:25}
  \end{align}
  By the linearity of the integral, it is enough to show that for
  \(k\) large enough, the following integral converges
  \begin{equation*}
    \int_{\bR} \sum_{n=k}^{\infty} (n+1)n!e^{-2(\lambda+\varepsilon)n}
    \abs*{\Phi^{(n)}(\cdot,t)}_{L^2(\bR^{n})}^2\,dt.
  \end{equation*}
  Note that for any \(\varepsilon>0\) there is a \(k\in\bN_0\) such
  that for any \(n\geq k\) we have
  \((n+1)e^{-2(\lambda+\varepsilon)n}<e^{-2\lambda n}\).  This follows
  from the fact that \(f(x)=\tfrac{x+1}{x}\) is strictly decreasing in
  the interval \((0,\infty)\) and \(\lim_{x\to\infty}f(x)=0\).  Hence,
  for \(k\) large enough, we have
  \begin{align*}
    \MoveEqLeft \int_{\bR} \sum_{n=k}^{\infty}
    (n+1)n!e^{-2(\lambda+\varepsilon)n}
    \abs*{\Phi^{(n)}(\cdot,t)}_{L^2(\bR^{n})}^2\,dt\\
    \leq{}& \int_{\bR} \sum_{n=k}^{\infty}n!e^{-2\lambda n}
    \abs*{\Phi^{(n)}(\cdot,t)}_{L^2(\bR^{n})}^2\,dt\\
    \leq{}& \int_{\bR} \sum_{n=0}^{\infty}n!e^{-2\lambda n}
    \abs*{\Phi^{(n)}(\cdot,t)}_{L^2(\bR^{n})}^2\,dt\\
    \leq{}&
    \int_{\bR} \norm*{\Phi(t)}_{-\lambda}^2\,dt
  \end{align*}
  Note that we can treat the first \( n \) elements of the sum in
  \cref{eq:25} as in the proof of \cref{thm:1}, so
  \(\norm*{\delta(\Phi)}_{-\lambda-\varepsilon}\leq
  C_{\varepsilon}\int_{\bR}\norm*{\Phi(t)}^2_{-\lambda}\,dt<\infty\),
  as required.
\end{proof}

It is a well known fact, that in the setting of Hida spaces
\((\sS),(\sS)^{*}\) the Skorohod integral can be interpreted as a
white noise integral.  Namely, for \(\Phi(t)\in(\sS)^{*}\) we can view
the following integral as the extension of the Skorohod integral
\begin{equation*}
  \int_{\bR}\partial^{*}_t\Phi(t)\,dt,
\end{equation*}
where the integral is understood in Pettis sense and
\(\partial_t^{*}:(\sS)^{*}\to(\sS)^{*} \) is the white noise
integration operator, that is the adjoint to \( \partial_t \), the
G\^ateaux derivative in the direction \(\delta_t\).  We have an
explicit expression for the chaos expansion of the above integral,
given by \citep[e.g.][]{K:1996,HKPS:1993}
\begin{equation}
  \label{eq:6}
  \int_{\bR}\partial^{*}_t\Phi(t)\,dt
  = \sum_{n=0}^{\infty}I_{n+1}
  \left( \int_{\bR}\delta_t\widehat{\otimes}\Phi^{(n)}(t)\,dt \right).
\end{equation}
As in the case of stochastic derivative, with the same notation as
previously, it is straightforward to check that for \(
\Phi^{(n)}(\cdot,t)\in L^2(\bR^n) \) we have
\begin{equation*}
  \int_{\bR}\delta_t\widehat{\otimes}\Phi^{(n)}(t)\,dt
  = \frac{1}{n}\sum_{k=0}^{n}\Phi^{(n)}(\underline{x}^{(n)}_{\not k},x_k)
  = \widehat{\Phi}^{(n)}(\cdot,t).
\end{equation*}
Thus this integral is an actual extension of the stochastic integral
defined in \cref{eq:5}.

Recall, that the same integral can be defined (in the \( (\sS)^{*} \)
setting) as
\begin{equation*}
  \int_{\bR}\Phi(t)\diamond W_t\,dt,
\end{equation*}
and the chaos expansion of this generalized random variable is the
same as the one in \cref{eq:6}. Note however, that \(
W_t=I_1(\delta_t)\in(\sS)^{*} \) is not an element of \( \sG^{*} \)
because \( \delta_t\notin L^2(\bR) \).  But the above reasoning
justifies \cref{def:1} as a Skorohod integral of processes in \(
\sG^{*} \) and \cref{thm:2} gives sufficient conditions for the result
of integration to be an element of \( \sG^{*} \).

\subsection{Properties of the stochastic integral}
\label{sec:prop-integral}
First, we state some properties that are readily seen directly from
\cref{def:1} of the Skorohod integral.

\begin{theorem}
  \begin{enumerate}
  \item The Skorohod integral is linear;
  \item \( \int_a^b0\,\delta B(t)=0 \);
  \item \( \int_a^b1\,\delta B(t) = B(b)-B(s) \);
  \item If \( a<b<c \) then \( \int_a^b\Phi(t)\,\delta B(t) +
    \int_b^c\Phi(t)\,\delta B(t) = \int_a^c\Phi(t)\,\delta B(t) \);
  \end{enumerate}
\end{theorem}

Next we present a ``fundamental theorem of calculus'' in our
setting. The proof of this result follows closely the proof in the
Malliavin calculus setting, which is natural, as the definitions
coincide on the intersection of the domains.

\begin{theorem}
  Suppose that \( \Phi(t)\in\sG^{*} \) is Skorohod-integrable over \(
  \sT \), and \( D_t\Phi(s) \) is Skorohod-integrable for almost all
  \( t\in\sT \). Then
\begin{equation}
  \label{eq:9}
  D_t \left( \int_{\sT} \Phi(s)\,\delta B(s) \right)
  = \Phi(t) + \int_{\sT} D_t\Phi(s) \,\delta B(s).
\end{equation}
\end{theorem}
\begin{proof}
  Note that since \( \Phi(t)\in\sG^{*} \) for all \( t\in\sT \), by
  \cref{thm:1} the stochastic derivative \( D_t\Phi(s) \) exists for
  almost all \( t\in\sT \) and its norm is square-integrable, hence \(
  D_s\Phi(s) \) is Skorohod-integrable by \cref{thm:2}.  It remains to
  show that \cref{eq:9} holds.

  Let \( \Phi(t) = \sum_{n=0}^{\infty} I_n \left( \Phi^{(n)}(\cdot,t)
  \right) \), where \( \Phi^n(\cdot,t)\in \widehat{L}^2(\bR^n) \) for
  all \( t\in\sT \). The left-hand side of \cref{eq:9} is given by
  \begin{align*}
    D_t \left( \int_{\sT} \Phi(s)\,\delta B(s) \right) ={}& D_t \left(
      \sum_{n=0}^{\infty}
      I_{n+1} \left( \widehat{\Phi}^{(n)} \right)\right)\\
    ={}& \sum_{n=0}^{\infty} (n+1) I_n \left( \Phi^{(n)}(\cdot,t)
    \right).
  \end{align*}
  On the other hand, we can write out the right-hand side of
  \cref{eq:9} as
  \begin{align*}
    \Phi(t) + \int_{\sT} D_t\Phi(s) \,\delta B(s) ={}&
    \sum_{n=0}^{\infty} I_n \left( \Phi^{(n)}(\cdot,t) \right) +
    \delta \left( D_t \left( \sum_{n=0}^{\infty}
        I_n \left( \Phi^{(n)}(\cdot,t) \right) \right) \right)\\
    ={}& \sum_{n=0}^{\infty} I_n \left( \Phi^{(n)}(\cdot,t) \right) +
    \delta \left( \sum_{n=0}^{\infty} n
      I_{n-1} \left( \Phi^{(n)}(\cdot,s,t) \right) \right)\\
    ={}& \sum_{n=0}^{\infty} I_n \left( \Phi^{(n)}(\cdot,t) \right) +
    \sum_{n=0}^{\infty} n I_n \left( \widehat{\Phi}^{(n)}(\cdot,t) \right) \\
    ={}& \sum_{n=0}^{\infty} I_n \left( \Phi^{(n)}(\cdot,t) \right) +
    \sum_{n=0}^{\infty} n I_n \left( \Phi^{(n)}(\cdot,t) \right) \\
    ={}& \sum_{n=0}^{\infty} (n+1)I_n \left( \Phi^{(n)}(\cdot,t)
    \right).
  \end{align*}
  So the two sides of \cref{eq:9} are equal and the theorem holds.
\end{proof}

When comparing the above result with its Malliavin calculus
counterpart \cite[see][Proposition 1]{BNBPV:2012}, we see that we are
not required to assume the existence of the stochastic derivative of
\( \Phi \) because it is ensured by the properties of the derivative
in the space \( \sG^{*} \).

Next, we present an ``integration by parts formula'' for the stochastic
derivative and integral. Note that we cannot use the pointwise product
freely as its result might be undetermined for generalized random
variables. However, we can always take a product of test and
generalized random variables.

\begin{theorem}
  \label{thm:3}
  Suppose that \( \phi\in\sG \) and \( \Phi(t)\in\sG^{*} \) for all \(
  t \). If for some \( \lambda>0 \) and \(
  \nu>\tfrac{1}{2}\ln(2+\sqrt{2}) \)
  \begin{equation*}
    \int_0^T \norm*{\Phi(t)}_{-\lambda+\nu}^2\,dt <\infty,
  \end{equation*}
  then
  \begin{equation}
    \label{eq:11}
    \int_0^T\phi\Phi(t)\,\delta B(t)
    = \phi\int_0^T\Phi(t)\,\delta B(t)-\int_0^T\Phi(t)D_t\phi\,dt.
  \end{equation}
\end{theorem}
\begin{proof}
  First we show that all components of \cref{eq:11} are elements of \(
  \sG^{*} \). By \cref{thm:2}, for the integral on the left-hand side
  of \cref{eq:11} to be well-defined it suffices that \(
  \int_0^T\norm*{\phi\Phi(t)}_{-\lambda}^2\,dt<\infty \). By
  \cref{thm:5} and our assumption, we have
  \begin{align*}
    \int_0^T\norm*{\phi\Phi(t)}_{-\lambda}^2\,dt \leq{}&
    C_{\nu}^2\norm*{\phi}_{\lambda}
    \int_0^T\norm*{\Phi(t)}_{-\lambda+\nu}^2\,dt\\
    <{}&\infty.
  \end{align*}
  Thus \( \phi\Phi(t) \) is Skorohod-integrable.

  The integral in the first component on the right-hand side of
  \cref{eq:11} is also well-defined by our assumption on
  square-integrability of the norm. Since \( \phi\in\sG \), the first
  product on the right-hand side is an element of \( \sG^{*} \).

  Finally, the Pettis integral on the right-hand side of \cref{eq:11}
  exists because for any \( \psi\in\sG \)
  \begin{align*}
    \MoveEqLeft \int_0^T\abs*{\ddangle*{\Phi(t)D_t\phi,\psi}}\,dt\\*
    \leq{}& \int_0^T \norm*{\Phi(t)D_t\phi}_{-\lambda+\varepsilon}
    \norm*{\psi}_{\lambda-\varepsilon} \,dt\\
    \leq{}& \norm*{\psi}_{\lambda-\varepsilon}C_{\nu} \int_0^T
    \norm*{\Phi(t)}_{-\lambda+\varepsilon+\nu}
    \norm*{D_t\phi}_{\lambda-\varepsilon}\,dt\\
    \leq{}& \norm*{\psi}_{\lambda-\varepsilon}C_{\nu}
    \left(\int_0^T\norm*{\Phi(t)}_{-\lambda+\varepsilon+\nu}^2\,dt\right)^{\frac{1}{2}}
    \left(\int_0^T\norm*{D_t\phi}_{\lambda-\varepsilon}^2\,dt\right)^{\frac{1}{2}}\\
    <{}&\infty.
  \end{align*}
  The first integral in the last statement above is finite by
  assumption and monotonicity of the norms \( \norm*{\cdot}_{-\lambda}
  \). The second integral is finite by \cref{thm:1}. Above we have
  also used \cref{thm:5,rem:2}.

  Finally, although very tedious, it is straightforward to check that
  the chaos expansions of both sides of \cref{eq:11} agree. In order
  to see this, one might start with \( \Phi=I_n(\Phi^{(n)}(t)) \) and
  \( \phi=I_{m}(\phi^{(m)}) \) as linear combinations of variables of
  this form are dense in \( \sG^{*} \) and \( \sG \)
  respectively. This choice of \( \phi,\Phi \) significantly
  simplifies the computations as one can use the product formula
  \begin{equation*}
    \phi\cdot\Phi
    = \sum_{n=0}^{\infty}\sum_{m=0}^{\infty}\sum_{k=0}^{m\wedge n}
    k!\binom{m}{k}\binom{n}{k}
    I_{m+n-2k}\left( \Phi^{(n)}\widehat{\otimes}_k\phi^{(m)} \right),
  \end{equation*}
  where \( \widehat{\otimes}_k \) is the symmetrized tensor product on
  \( k \) variables.
\end{proof}

The last well-known property of the Skorohod integral that we use in
the forthcoming sections is the form of its \( \sS \)-transform.
\begin{proposition}
  \label{prop:7-3}
  For all \( \Phi\in(\sS)^{*} \) and \( \xi\in\sS(\bR) \),
  \[
  \sS \left( \int_0^t \Phi(s)\,\delta B(s) \right) = \int_0^t\sS
  \left( \Phi(s) \right)(\xi)\cdot\xi(s)\,ds.
  \]
\end{proposition}

\section{Integration for Volterra processes}
\label{sec:integration-vmbv}
As we have already mentioned, in order to define an integral with
respect to \vmbv{} process, we follow \citet{BNBPV:2012}. We define
the integral
\begin{equation}
  \label{eq:int-1}
  \int_0^t \Phi(s)\,dX_1(s),
  \qquad\text{where}\qquad
  X_1(t)=\int_0^tg(t,s)\,dB(s),
\end{equation}
with the use of the following operator
\begin{equation}
  \label{eq:K}
  \sK_g(\Phi)(t,s)
  \eqdef \Phi(s)g(t,s) + \int_s^t \left(\Phi(u)-\Phi(s)\right)\,g(du,s).
\end{equation}
The definition of the integral in \cref{eq:int-1} is given by
\begin{equation}
  \label{eq:4}
  \int_0^t\Phi(s)\,dX_1(s)
  \eqdef
  \int_0^t \sK_g(\Phi)(t,s)\,\delta B(s)
  + \int_0^t D_s\left\{\sK_g(\Phi)(t,s)\right\}\,ds.
\end{equation}

Before we discuss the integral defined above, we have to turn our
attention to the study of the properties of the operator \( \sK_g \)
which is a main building block of the integral itself.

\subsection{Properties of the operator
  \texorpdfstring{\(\sK_{g}\)}{Kg}}
\label{sec:prop-K}
In this section, we study the regularity of the operator \(\sK_g\).
That is we wish to find out for which \(\gamma>0\) does
\(\sK_{g}(\Phi)(t,s)\in\sG_{-\gamma}\) when
\(\Phi(t)\in\sG_{-\lambda}\) for all \(t\). First of all, from
\cref{eq:K}, we see that \(\Phi(u)-\Phi(s)\) has to be
Pettis--Stieltjes-integrable with respect to \(g(du,s)\) on \([s,t]\)
for \(0\leq s < t \leq T\).  Using previously introduced notation, we
have \( (\sT,\sB,m) = ([s,t],\sB([s,t]),m_g) \), where \( m_g \) is
the Lebesgue--Stieltjes measure associated to \( g(\cdot,s) \).  In
order to consider integrability of \( \Phi \) with respect to \(
g(du,s) \), and later, for the existence of \( \sK_g(\Phi)(t,s) \) we need
the following assumptions.

\begin{assumption}
  \label{as:A}
Suppose that
  \begin{enumerate}
  \item \label{item:1} For any \(0\leq s<u<v <T\) the function
    \(u\mapsto g(u,s)\) is of bounded variation on \([u,v]\);
  \item \label{item:2} The mapping \([0,T]\ni t\mapsto\Phi(t)\in\sG^*\) is
    weakly measurable;
  \item \label{item:3} For any \( 0\leq s\leq t\leq T \)
  \begin{equation*}
    \int_s^t\norm*{\Phi(u)-\Phi(s)}^2_{-\lambda}\, \abs{g}(du,s)<\infty.
  \end{equation*}
  \end{enumerate}
\end{assumption}

\cref{as:A} \cref{item:1} ensures that we can define a
Pettis--Stieltjes integral with respect to \(g(du,s)\).  Under
\cref{as:A} \cref{item:2}, the mapping \(u\mapsto \Phi(u)-\Phi(s)\) is
weakly measurable, as are all mappings considered in the remainder of
this paper.

\begin{proposition}
  \label{prop:1}
  Under \cref{as:A}, the integral
  \begin{equation}
    \label{eq:3}
    \int_s^t(\Phi(u)-\Phi(s))\,g(du,s)
  \end{equation}
  exists as a Pettis--Stieltjes integral. Moreover, if \(
  \Phi(t)\in\sG_{-\lambda},\lambda>0 \) for all \( 0\leq t\leq T \),
  then for any \(0\leq s <t \leq T\),
  \begin{equation}
    \label{eq:12}
    \norm*{\int_s^t(\Phi(u)-\Phi(s))\,g(du,s)}_{-\lambda}<\infty,
  \end{equation}
  that is the integral in \cref{eq:3} is an element of
  \(\sG_{-\lambda}\).
\end{proposition}
\begin{proof}
  In order to prove integrability in the Pettis sense, consider
  \begin{align*}
    \MoveEqLeft
\abs*{\ddangle*{\int_s^t(\Phi(u)-\Phi(s))\,g(du,s),\phi}}\\
={}&\abs*{\int_s^t\ddangle*{(\Phi(u)-\Phi(s)),\phi}\,g(du,s)}\\
\leq{}&    \int_s^t\abs*{\ddangle*{ \Phi(u)-\Phi(s),\phi}}\,\abs*{g}(du,s)\\
    \leq{}&
    \int_s^t\norm*{\Phi(u)-\Phi(s)}_{-\lambda}\norm*{\phi}_\lambda\,\abs*{g}(du,s)\\
    ={}&
    \norm*{\phi}_\lambda\int_s^t\norm*{\Phi(u)-\Phi(s)}_{-\lambda}\,\abs*{g}(du,s)\\
    \leq{}&
    \norm*{\phi}_\lambda
    \left(\int_s^t\norm*{\Phi(u)-\Phi(s)}_{-\lambda}^2\,\abs*{g}(du,s)\right)^{\frac{1}{2}}
    \left(\int_s^t1\,\abs*{g}(du,s)\right)^{\frac{1}{2}}\\
    <{}&
    \infty,
\end{align*}
where we have used the H\"older inequality, \cref{as:A,rem:2}.

To prove that the norm in \cref{eq:12} is finite, consider
  \begin{align*}
    \MoveEqLeft\norm*{\int_s^t(\Phi(u)-\Phi(s))\,g(du,s)}_{-\lambda}^2\\
    ={}&
    \sum_{n=0}^\infty n!\abs*{\int_s^t (\Phi^{(n)}(u)-\Phi^{(n)}(s))\,g(du,s)}^2_{-\lambda}\\
    \leq{}& V_s^t\left[g(\cdot,s)\right]\sum_{n=0}^\infty n!
    \int_s^t \abs*{\Phi^{(n)}(u)-\Phi^{(n)}(s)}^2_{-\lambda}\,\abs*{g}(du,s)\\
    ={}& V_s^t\left[g(\cdot,s)\right]\int_s^t\sum_{n=0}^\infty n!
    \abs*{\Phi^{(n)}(u)-\Phi^{(n)}(s)}^2_{-\lambda}\,\abs*{g}(du,s)\\
    ={}&
    V_s^t\left[g(\cdot,s)\right]\int_s^t\norm*{\Phi(u)-\Phi(s)}_{-\lambda}^2\,\abs*{g}(du,s)\\
    <{}&
    \infty,
  \end{align*}
  where \(V_s^t[f]\) denotes the total variation of \(f\) on the
  interval \([s,t]\), which by \cref{as:A} is finite.
\end{proof}

\begin{theorem}
  If \cref{as:A} holds and \( \Phi(t)\in\sG_{-\lambda} \) for all \(
  0\leq t\leq T \), then \( \sK_g(\Phi)(t,s)\in\sG_{-\lambda} \) for
  all \( 0\leq s\leq t\leq T \).
\end{theorem}
\begin{proof}
  As in the proof of the \cref{prop:1}, it is enough to establish
  that \( \norm*{\sK_g(\Phi)(t,s)}_{-\lambda}<\infty \). Consider
  \begin{align*}
    \norm*{\sK_g(\Phi)(t,s)}_{-\lambda} ={}&
    \norm*{ \Phi(s)g(t,s)+\int_s^t(\Phi(u)-\Phi(s))\,g(du,s) }_{-\lambda} \\
    \leq{}& \norm*{ \Phi(s)g(t,s)}_{-\lambda}
    + \norm*{\int_s^t(\Phi(u)-\Phi(s))\,g(du,s) }_{-\lambda}\\
    ={}& \abs*{g(t,s)}\norm*{ \Phi(s)}_{-\lambda}
    + \norm*{\int_s^t(\Phi(u)-\Phi(s))\,g(du,s) }_{-\lambda}\\
    <{}&\infty.
  \end{align*}
  Thus the result holds
\end{proof}

As we will see in the forthcoming sections, the fact that the operator
\( \sK_g(\cdot) \) preserves the regularity of \( \Phi \) is of
crucial importance in the derivation of regularity properties of the
integrals defined below.

\subsection{The integral}
\label{sec:integral-vmbv}
Now we go back to the study of the integral defined in
\cref{eq:int-1}.  Since we have established sufficient conditions for
\( \sK_g(\Phi)(t,s)\in\sG_{-\lambda} \), we can now look at the
Skorohod integral of \( \sK_g(\Phi)(t,s) \). By \cref{thm:2}, it is
enough to show that
\begin{equation*}
  \int_0^T\norm*{\sK_g(\Phi)(t,s)}_{-\lambda}^2\,ds<\infty
\end{equation*}
in order to establish Skorohod integrability of \( \sK_g(\Phi)(t,s)
\).  We will show that this is the case under the following
assumptions.
\begin{assumption}
  \label{as:B}
Suppose that
  \begin{enumerate}
  \item \label{item:4}
    \begin{equation*}
      \int_0^T\abs*{g(t,s)}^2\norm*{\Phi(s)}_{-\lambda}^2\,ds<\infty;
    \end{equation*}
  \item \label{item:5} For any \(0\leq s <t <T\)
    \begin{equation*}
      \int_0^t\norm*{\int_s^t(\Phi(u)-\Phi(s))\,g(du,s)}_{-\lambda}^2\,ds<\infty.
    \end{equation*}
  \end{enumerate}
\end{assumption}

\begin{remark}
  \label{rem:1}
  Notice that in what follows, \cref{as:B} \cref{item:4,item:5} can be
  substituted with the weaker assumption that \(
  \int_0^t\norm*{\sK_g(\Phi)(t,s)}_{-\lambda}^2\,ds<\infty \) for all
  \( t\in[0,T] \).
\end{remark}

\begin{proposition}
  \label{prop:5}
  Suppose that \cref{as:A,as:B} hold. If \( \Phi(t)\in\sG_{-\lambda}
  \) for all \( 0 \leq t\leq T \), then for any \(\varepsilon>0\),
  \begin{equation*}
    \int_0^t \sK_g(\Phi)(t,s)\,\delta B(s)\in\sG_{-\lambda-\varepsilon}.
  \end{equation*}
\end{proposition}
\begin{proof}
  Using our assumptions and H\"older's inequality we obtain
  \begin{align*}
    \MoveEqLeft\int_0^t\norm*{\sK_g(\Phi)(t,s)}_{-\lambda}^2\,ds\\*
    ={}&
    \int_0^t\norm*{\Phi(s)g(t,s)+\int_s^t(\Phi(u)-\Phi(s))\,g(du,s)}_{-\lambda}^2\,ds\\
    \leq{}&
    \int_0^t\left(\norm*{\Phi(s)g(t,s)}_{-\lambda}
      +\norm*{\int_s^t(\Phi(u)-\Phi(s))\,g(du,s)}_{-\lambda}\right)^2\,ds\\
    \leq{}&
    2\int_0^t\norm*{\Phi(s)g(t,s)}_{-\lambda}^2\,ds
    +2\int_0^t \norm*{\int_s^t(\Phi(u)-\Phi(s))\,g(du,s)}_{-\lambda}^2\,ds\\
    <{}&\infty.
  \end{align*}
  Hence the result follows by \cref{thm:2}.
\end{proof}

Next, we consider Pettis-integrability of \( D_s\sK_g(\Phi)(t,s) \).

\begin{proposition}
  \label{prop:4}
  Suppose that \cref{as:A,as:B} hold. If \( \Phi(t)\in\sG_{-\lambda}
  \), then for any \( \varepsilon>0 \),
  \begin{equation*}
    \int_0^t D_s\sK_g(\Phi)(t,s)\,ds \in \sG_{-\lambda-\varepsilon}.
  \end{equation*}
\end{proposition}
\begin{proof}
  We will show that \( D_s\sK_g(\Phi)(t,s) \) is weakly in \(
  L^1([0,t]) \), that is for all  \( \phi\in\sG \) we have  \( \ddangle*{D_s\sK_g(\Phi)(t,s),\phi}\in L^1([0,t]) \) .  Observe that if \( \Phi(t)\in\sG_{-\lambda} \), then
  \( \sK_g(\Phi)(t,s)\in\sG_{-\lambda} \) and in consequence \(
  D_s\sK_g(\Phi)(t,s)\in\sG_{-\lambda-\varepsilon} \) for any \(
  \varepsilon>0 \).  Consider
  \begin{align*}
    \int_0^t\abs*{\ddangle*{D_s\sK_g(\Phi)(t,s),\phi}}\,ds \leq{}&
    \int_0^t\norm*{D_s\sK_g(\Phi)(t,s)}_{-\lambda-\varepsilon}
    \norm*{\phi}_{\lambda+\varepsilon}\,ds\\
    ={}& \norm*{\phi}_{\lambda+\varepsilon}
    \int_0^t\norm*{D_s\sK_g(\Phi)(t,s)}_{-\lambda-\varepsilon}\,ds\\
    \leq{}& \norm*{\phi}_{\lambda+\varepsilon} t
    \int_0^t\norm*{D_s\sK_g(\Phi)(t,s)}_{-\lambda-\varepsilon}^2\,ds\\
    <{}& \infty.
  \end{align*}
  Above we have used \cref{rem:2}, H\"older's inequality and
  \cref{thm:1}.
\end{proof}

Putting \cref{prop:5,prop:4} together yields the main result of this
section.

\begin{theorem}
  \label{thm:4}
  Suppose that \cref{as:A,as:B} hold. If \( \Phi(t)\in\sG_{-\lambda}
  \) for all \( t\in[0,T] \), then for any \( \varepsilon>0 \)
  \begin{equation*}
    \int_0^t\Phi(s)\,dX_1(s) \in \sG_{-\lambda-\varepsilon}.
  \end{equation*}
\end{theorem}

\begin{remark}
  Recall that \(
  \sG\subset\sG_{\lambda}\subset(L^2)\subset\sG_{-\lambda}\subset\sG^{*}
  \) for any \( \lambda>0 \). So the above theorem assures that
  \begin{enumerate}
  \item if \( \Phi\in\sG_{\lambda} \) for some \( \lambda>0 \) then
    the integral is an \( (L^2) \) process;
  \item if \( \Phi\in\sG \) then the integral is a \( \sG \) process
    again;
  \item if \( \Phi\in(L^2) \), then the integral is a \(
    \sG_{-\varepsilon} \) process for any \( \varepsilon>0 \), thus in
    a certain sense it is very close to \( (L^2) \).
  \end{enumerate}
\end{remark}

\section{Integration for volatility modulated Volterra processes}
\label{sec:integration-volatility}
In this section, we introduce stochastic volatility in the integrator
process \( X(t) \).  In the defining \cref{eq:10} we see that the
volatility is multiplying the integrands on the right-hand side of
\cref{eq:10}.  This is an ordinary operation when considering
non-generalized stochastic processes, however the product of two
generalized random variables from \( \sG^{*} \) does not have to be an
element of \( \sG^{*} \).  We overcome this difficulty in two
different approaches.  In the first part of this section, we take \(
\Sigma(s)=\sigma(s) \) to be a test stochastic process, that is \(
\sigma(s)\in\sG \) for all \( s\in[0,T] \).  In the second part of
this section, we use the Wick product to introduce volatility
modulation as this operation is well defined for all \(
\Sigma\in\sG^{*} \). Note that under strong independence (see \cref{def:2}
or \citet{B:2001,BP:1996}) this is equivalent to the pointwise product
case.

\subsection{Pointwise product with smooth volatility}
\label{sec:volatility-smooth}
In this subsection, we assume that the volatility process \( \sigma
\) is a smooth stochastic process and study the following integral
\begin{equation}
  \label{eq:int-s}
  \int_0^t \Phi(s)\,dX_\sigma(s),
  \qquad\text{where}\qquad
  X_\sigma(t)=\int_0^tg(t,s)\sigma(s)\,dB(s).
\end{equation}

Let us remark that the assumption that the volatility is a stochastic
test process is not as restrictive as it appears. For example,
Brownian-driven Volterra processes are test stochastic processes because
\begin{equation*}
  \sigma(t)
  = \int_0^th(t,s)\,dB(s)
  = I_1 \left( \mathbbm{1}_{[0,t]}h(t,\cdot) \right).
\end{equation*}
So if \( h(t,\cdot)\in L^2(\bR) \), then \(
\norm*{\sigma(t)}_{\lambda}<\infty \) for all \( t\in[0,T] \) and \(
\lambda>0 \), hence \( \sigma(t)\in\sG \) for all \( t\in[0,T] \).

As we will show, the sufficient conditions for the integral in
\cref{eq:int-s} to be well-defined are the following.
\begin{assumption}
  \label{as:C}
  Suppose that
  \begin{enumerate}
  \item \label{item:6} For all \( t\in[0,T] \) we have \(
    \sigma(t)\in\sG \);
  \item \label{item:7}
    \begin{equation*}
      \int_0^T\norm*{\sigma(s)}^2_{\lambda}\,ds<\infty;
    \end{equation*}
  \item\label{item:8} For any \( 0\leq s< t\leq T \)
    \begin{equation*}
      \int_0^t\abs*{g(t,s)}^2\norm*{\Phi(s)}^2_{-\lambda+\nu}
      \norm*{\sigma(s)}^2_{-\lambda}\,ds<\infty.
    \end{equation*}
  \item\label{item:9} For any \( 0\leq s< t\leq T \)
    \begin{equation*}
      \int_0^t\norm*{\int_s^t (\Phi(u)-\Phi(s))\,g(du,s)}^2_{-\lambda+\nu}
      \norm*{\sigma(s)}^2_{-\lambda}\,ds<\infty.
    \end{equation*}
  \end{enumerate}
\end{assumption}

\begin{remark}
  \label{rem:3}
  As previously, in what follows, \cref{as:C} \cref{item:8,item:9} can
  be substituted with the weaker assumption that \( \int_0^t
  \norm*{\sK_g(\Phi)(t,s)}_{-\lambda+\nu}^2
  \norm*{\sigma(s)}_{\lambda}^2\,ds<\infty \) for all \( t\in[0,T] \).
\end{remark}

\begin{theorem}
  Under \cref{as:A,as:C} the integral
  \begin{equation}
    \label{eq:7}
    \int_0^t \Phi(s)\,dX_{\sigma}(s)
  \end{equation}
  is well-defined in the sense of Pettis. Moreover, if \(
  \Phi(t)\in\sG_{-\lambda+\nu} \) where \(
  \nu>\tfrac{1}{2}\ln(2+\sqrt{2}) \), then for any \( \varepsilon>0 \)
  \begin{equation*}
    \int_0^t\Phi(s)\,dX_{\sigma}(s)\in\sG_{-\lambda-\varepsilon}.
  \end{equation*}
\end{theorem}
\begin{proof}
  First we establish the existence of the Skorohod integral.
  By \cref{thm:2} it suffices to show that
  \begin{equation*}
    \int_0^t\norm*{\sK_g(\Phi)(t,s)\cdot\sigma(s)}_{-\lambda}^2\,ds <\infty.
  \end{equation*}
  This follows immediately from \cref{thm:5,as:C}:
  \begin{align*}
    \MoveEqLeft\int_0^t\norm*{\sK_g(\Phi)(t,s)\cdot\sigma(s)}_{-\lambda}^2\,ds\\
    \leq{}& C_{\nu}^2
    \int_0^t\norm*{\sK_g(\Phi)(t,s)}_{-\lambda+\nu}^2
    \norm*{\sigma(s)}_{\lambda}^2\,ds\\
    \leq{}& 2C_{\nu}^2
    \int_0^t\abs{g(t,s)}^2\norm*{\Phi(s)}_{-\lambda+\nu}^2\norm*{\sigma(s)}_{\lambda}^2\,ds\\
    &+
    2C_{\nu}^2\int_0^t\norm*{\int_s^t\Phi(u)- \Phi(s)\,g(du,s)}_{-\lambda+\nu}^2\norm*{\sigma(s)}_{\lambda}^2\,ds\\
    <{}&\infty.
  \end{align*}

  The existence of the Pettis integral follows from
  \cref{rem:2,thm:5,as:C,thm:1}.  We have to show that \(
  \ddangle*{D_s \left( \sK_g(\Phi)(t,s) \right)\cdot\sigma(s),\phi} \)
  is integrable for any \( \phi\in\sG \):
  \begin{align*}
    \MoveEqLeft\int_0^t\abs*{\ddangle*{D_s
        \left( \sK_g(\Phi)(t,s) \right)\cdot\sigma(s),\phi}}\,ds\\
    \leq{}& \int_0^t\norm*{D_s \sK_g(\Phi)(t,s) \cdot
      \sigma(s)}_{-\lambda}
    \norm*{\phi}_{\lambda}\,ds\\
    \leq{}& C_{\nu}\norm*{\phi}_{\lambda} \int_0^t\norm*{D_s
      \sK_g(\Phi)(t,s)}_{-\lambda+\nu}
    \norm*{\sigma(s)}_{\lambda}\,ds\\
    \leq{}& C_{\nu} \norm*{\phi}_{\lambda} \left(\int_0^t\norm*{D_s
        \sK_g(\Phi)(t,s)}_{-\lambda+\nu}^2\,ds\right)^{\frac{1}{2}}
    \left(\int_0^t \norm*{\sigma(s)}_{\lambda}^2\,ds\right)^{\frac{1}{2}}\\
    <{}&\infty.
  \end{align*}
  Finally, suppose that \( \Phi(t)\in\sG_{-\lambda+\nu} \). Then \(
  \sK_g(\Phi)(t,s)\in\sG_{-\lambda+\nu} \) and consequently, \(
  \sK_g(\Phi)(t,s)\cdot\sigma(s)\in\sG_{-\lambda} \). Thus for any \(
  \varepsilon>0 \) we have \( \delta \left( \sK_g(\Phi)(t,s) \cdot
    \sigma(s) \right) \in \sG_{-\lambda-\varepsilon} \). Also, \( D_s
  \sK_g(\Phi)(t,s) \in \sG_{-\lambda+\nu-\varepsilon} \) and so \( D_s
  \sK_g(\Phi)(t,s) \cdot \sigma(s) \in \sG_{-\lambda-\varepsilon} \),
  hence \( \int_0^t D_s \sK_g(\Phi)(t,s) \cdot
  \sigma(s)\,ds\in\sG_{-\lambda-\varepsilon} \). So the theorem holds.
\end{proof}

In comparison with the results of \citet{BNBPV:2012}, the above result
allows integration of a larger class of processes, however it
restricts the class of volatility modulators. We present the extension
of the latter class in the next subsection.

\subsection{Wick product with generalized volatility}
\label{sec:volatility-wick}
Below, we consider the generalized stochastic process as the
volatility. Since the volatility is introduced through multiplication
and the pointwise product of two generalized stochastic processes does
not have to be well-defined, we use the Wick product instead. It is
worth noting that the choice between the pointwise and Wick products
should be based on modeling considerations as the two products
coincide only under special circumstances. We give an example of an
additional assumption on the volatility process that ensures the
equality of the Wick and pointwise products in the definition given
below.

We define the integral with respect to a Wick--\vmbv{} as
\begin{equation}
  \label{eq:int-W} \int_0^t\Phi(s)\,dX_{\diamond\Sigma}(s)
  \eqdef\int_0^t\sK_g(\Phi)(t,s)\diamond\Sigma(s)\,\delta B(s)
  + \int_0^tD_s \left( \sK_g(\Phi)(t,s) \right)\diamond\Sigma(s)\,ds,
\end{equation}
where \(X_{\diamond\Sigma}(t)=\int_0^tg(t,s)\Sigma(s)\,\delta
B(s)\). In what follows, we show that the following are sufficient
conditions for the integral in \cref{eq:int-W} to be well-defined.
\begin{assumption}
  \label{as:D}
  Suppose that
  \begin{enumerate}
  \item\label{item:10}
  \begin{equation*}
    \int_0^T \norm*{\Sigma(s)}_{-\lambda}^2\,ds<\infty;
  \end{equation*}
  \item\label{item:11} For any \( 0\leq s< t\leq T \)
    \begin{equation*}
      \int_0^t\abs*{g(t,s)}^2\norm*{\Phi(s)}^2_{-\lambda}
      \norm*{\Sigma(s)}^2_{-\lambda}\,ds<\infty.
    \end{equation*}
  \item\label{item:12} For any \( 0\leq s< t\leq T \)
    \begin{equation*}
      \int_0^t\norm*{\int_s^t (\Phi(u)-\Phi(s))\,g(du,s)}^2_{-\lambda}
      \norm*{\Sigma(s)}^2_{-\lambda}\,ds<\infty.
    \end{equation*}
  \end{enumerate}
\end{assumption}

\begin{remark}
  \label{rem:4}
  As in \cref{rem:1,rem:3}, in what follows, \cref{as:D}
  \cref{item:11,item:12} can be substituted with the weaker assumption
  that for all \( t\in[0,T] \) we have \( \int_0^t
  \norm*{\sK_g(\Phi)(t,s)}_{-\lambda}^2
  \norm*{\Sigma(s)}_{-\lambda}^2\,ds<\infty
  \).
\end{remark}

\begin{theorem}
  Under \cref{as:A,as:D} the integral in \cref{eq:int-W} is
  well-defined.  Moreover, if \( \Phi(t)\in\sG_{-\lambda} \) for all
  \( t\in[0,T] \), then for any \( \varepsilon>0 \)
  \begin{equation*}
    \int_0^t\Phi(s)\,dX_{\diamond\Sigma}(s) \in \sG_{-\lambda-\frac{1}{2}-\varepsilon}.
  \end{equation*}
\end{theorem}
\begin{proof}
  For the Skorohod integral in \cref{eq:int-W} to exist, it is enough
  to show that
  \begin{equation*}
    \int_0^t\norm*{\sK_g(\Phi)(t,s)\diamond\Sigma(s)}^2_{-\lambda}\,ds
    <\infty.
  \end{equation*}
  Applying \cref{prop:3}, with \( \varepsilon>0 \) we have
  \begin{align*}
    \MoveEqLeft \int_0^t
    \norm*{\sK_g(\Phi)(t,s)\diamond\Sigma(s)}^2_{-\lambda-\frac{1}{2}-\varepsilon}
    \,ds\\*
    \leq{}&
    C_{\varepsilon}^2\int_0^t\norm*{\sK_g(\Phi)(t,s)}^2_{-\lambda}
    \norm*{\Sigma(s)}^2_{-\lambda}\,ds\\
    \leq{}&
    C_{\varepsilon}^2\int_0^t\abs*{g(t,s)}^2\norm*{\Phi(s)}^2_{-\lambda}
    \norm*{\Sigma(s)}^2_{-\lambda}\,ds\\
    &+C_{\varepsilon}^2\int_0^t \norm*{\int_s^t
      (\Phi(u)-\Phi(s))\,g(du,s)}^2_{-\lambda}
    \norm*{\Sigma(s)}^2_{-\lambda}\,ds\\
    <{}&\infty.
\end{align*}
Thus the Skorohod integral in \cref{eq:int-W} exists and is an element
of \( \sG_{-\lambda-\frac{1}{2}-\varepsilon} \).

Now, using arguments similar to the ones used in the case of \(
\sigma=1 \), we show that under \cref{as:D} the Pettis integral in
\cref{eq:int-W} also exists.  Below we apply \cref{prop:3}.  For any
\( \phi\in\sG \) and \( \varepsilon>0 \) consider
\begin{align*}
  \MoveEqLeft\int_0^t\abs*{\ddangle*{D_s \left( \sK_g(\Phi)(t,s)
      \right)
      \diamond\Sigma(s),\phi }}\,ds\\
  \leq{}& \int_0^t\norm*{D_s \left( \sK_g(\Phi)(t,s)
    \right)\diamond\Sigma(s)}_{-\lambda-\frac{1}{2}-\varepsilon}
  \norm*{\phi}_{\lambda+\frac{1}{2}+\varepsilon}\,ds\\
  \leq{}& C_{\varepsilon}
  \norm*{\phi}_{\lambda+\frac{1}{2}+\varepsilon}
  \int_0^t\norm*{D_s\sK_g(\Phi)(t,s)}_{-\lambda-\varepsilon}
  \norm*{\Sigma(s)}_{-\lambda-\varepsilon}\,ds\\
  \leq{}& C_{\varepsilon}
  \norm*{\phi}_{\lambda+\frac{1}{2}+\varepsilon}
  \left(\int_0^t\norm*{D_s\sK_g(\Phi)(t,s)}_{-\lambda-\varepsilon}^2\,ds\right)^{\frac{1}{2}}
  \left(\int_0^t
    \norm*{\Sigma(s)}_{-\lambda-\varepsilon}^2\,ds\right)^{\frac{1}{2}}\\
  <{}&\infty.
\end{align*}
The finiteness of the first integral above is a consequence of
\cref{thm:1} and finiteness of the second integral is ensured by
\cref{as:D} \cref{item:10}.
\end{proof}

Recall from \cref{thm:6} that if \( \Phi,\Psi\in\sG^{*} \) are
strongly independent, then \( \Phi\cdot\Psi=\Phi\diamond\Psi \). Using
this fact, we see that under an additional assumption of strong
independence of \( \sK_g(\Phi) \) and \( \Sigma \), we have the
following result.
\begin{corollary}
\label{cor:1}
  Suppose that \( \sK_g(\Phi)(t,s) \) and \( \Sigma(s) \) are strongly
  independent for all \( 0\leq s\leq t\leq T \). Under
  \cref{as:A,as:D} the integral
  \begin{equation}
    \label{eq:int-S}
    \int_0^t\Phi(s)\,dX_{\Sigma}(s)
    \eqdef\int_0^t\sK_g(\Phi)(t,s)\cdot\Sigma(s)\,\delta B(s)
    + \int_0^tD_s \left( \sK_g(\Phi)(t,s) \right)\cdot\Sigma(s)\,ds
  \end{equation}
  is well-defined and equal to the one in \cref{eq:int-W}.  Moreover,
  if \( \Phi(t)\in\sG_{-\lambda} \) for all \( t\in[0,T] \), then for
  any \( \varepsilon>0 \)
  \begin{equation*}
    \int_0^t\Phi(s)\,dX_{\Sigma}(s) \in \sG_{-\lambda-\frac{1}{2}-\varepsilon}.
  \end{equation*}
\end{corollary}

Note that we cannot assume that \( \Phi(t) \) and \( \Sigma(t) \) are
strongly independent as the operator \( \sK_g(\cdot) \) does not
preserve the support, that is \( \supp \left( \Phi^{(n)} \right)\neq
\supp \left( \sK_g(\Phi^{(n)}) \right)\).
However, in applications one often works with the volatility that is an \( (L^2) \) process independent of ``everything else'' in the model.
This case is covered by \cref{cor:1} and it turns out that we can interchange the Wick and the pointwise product making this a very flexible setup.
Observe that this case is in general not applicable in the setup of the previous section, as the space \( \sG \) is much smaller than the space \( (L^2) \).
Thus this extension of the class of volatility modulators is important.

\section{Properties of the integral}
\label{sec:prop-integral-vmbv}
First of all, recall that the definition of the stochastic derivative
is the same as the definition of the Malliavin derivative that is used
in \citet{BNBPV:2012} and the only difference is the domain of the
derivative.  Also, the Skorohod integral in our setting is defined
through the same formula as the Skorohod integral in
\citet{BNBPV:2012} but on a larger domain.  These two observations
allow us to state the following.
\begin{proposition}
  The integrals defined by \cref{eq:int-1,eq:int-s,eq:int-S}, and the
  one defined in \citet{BNBPV:2012} are equal on the intersection of
  their domains.
\end{proposition}
\begin{proof}
  This follows immediately from the definition of the stochastic
  derivative and Skorohod integral that we use and the fact that the
  Pettis integral is an extension of the Lebesgue integral to Banach
  space valued integrands.
\end{proof}
\begin{proposition}
  The integrals defined in \cref{eq:int-1,eq:int-s,eq:int-S,eq:int-W}
  are all linear.
\end{proposition}
\begin{proof}
  First observe that directly from the definition of the stochastic
  derivative and Skorohod integral we know that both of these
  operations are linear.  This, with the linearity of operator \(
  \sK_g \) and the fact that \( (a\Phi)\diamond \Psi =
  a(\Phi\diamond\Psi) \) and \( (\Phi+\Psi)\diamond\Sigma =
  \Phi\diamond\Sigma+\Psi\diamond\Sigma \) gives us linearity of the
  integral in all the cases considered above.
\end{proof}

\begin{proposition}
  If \( \Phi \) is integrable with respect to \( X_1 \), (\(
  X_{\sigma},X_{\diamond\Sigma},X_{\Sigma} \) respectively) on the
  interval \( [0,T] \) then for any \( S\in[0,T] \) it is also
  integrable on the interval \( [0,S] \). Moreover, the following
  holds
  \begin{equation*}
    \int_0^T \Phi(t)\mathbbm{1}_{[0,S]}(t)\,dX_{*}(t)
    =\int_0^S \Phi(t)\,dX_{*}(t),
  \end{equation*}
where \( *\in\{1,\sigma,\Sigma,\diamond\Sigma\} \)
\end{proposition}

\begin{theorem}
  Suppose that \( \phi\in\sG \) and \( \Phi(t) \) is \( dX_1 \) or \(
  dX_{\sigma} \)-integrable on \( [0,T] \). Then for \( t\in[0,T] \)
  \begin{equation}
    \label{eq:14}
    \int_0^t\phi\cdot\Phi(t)\,dX_{*} = \phi\cdot\int_0^t\Phi(t)\,dX_{*},
  \end{equation}
  where \( *\in\{1,\sigma\} \).
\end{theorem}
\begin{proof}
  Our arguments follow closely those in the proof of \cite[Proposition
  8]{BNBPV:2012}. First, note that the case with \( \sigma(s)\neq1 \)
  will differ from the one with \( \sigma(s)=1 \) only in the norm
  estimates as seen in the previous sections. It is enough to
  establish that \cref{eq:14} holds in one of the cases. Observe that
  \begin{equation}
    \label{eq:13}
    \sK_g(\phi\cdot\Phi)(t,s) = \phi\cdot\sK_g(\Phi)(t,s).
  \end{equation}
  Next, by \cref{eq:13,prop:6,thm:3}, we have
  \begin{align*}
    \MoveEqLeft\int_0^t\phi\Phi(s)\,dX_{\sigma}(s)\\*
    ={}& \int_0^t \sK_g(\phi\Phi)(t,s) \sigma(s) \,\delta B(s)
    + \int_0^t D_{s}\left\{\sK_g(\phi\Phi)(t,s)\right\} \sigma(s) \,ds \\
    ={}& \int_0^t \phi\sK_g(\Phi)(t,s) \sigma(s) \,\delta B(s)
    + \int_0^t D_{s}\left\{\phi\sK_g(\Phi)(t,s)\right\} \sigma(s) \,ds \\
    ={}&
    \int_0^t\phi\sK_g(\Phi)(t,s) \sigma(s) \,\delta B(s) \\
    &+ \int_0^t\left(\phi D_{s}\left\{\sK_g(\Phi)(t,s)\right\}
      + \sK_g(\Phi)(t,s) D_{s}\left\{\phi\right\}\right) \sigma(s) \,ds \\
    ={}& \phi\int_0^t\sK_g(\Phi)(t,s) \sigma(s) \,\delta B(s)
    - \int_0^tD_s\{\phi\} \sK_g(\Phi)(t,s) \sigma(s) \,ds\\
    &+ \phi\int_0^t D_{s}\left\{\sK_g(\Phi)(t,s)\right\} \sigma(s)
    \,ds
    + \int_0^t\sK_g(\Phi)(t,s) D_{s}\left\{\phi\right\} \sigma(s) \,ds \\
    ={}& \phi\int_0^t\sK_g(\Phi)(t,s) \sigma(s) \,\delta B(s)
    + \phi\int_0^t D_{s}\left\{\sK_g(\Phi)(t,s)\right\} \sigma(s) \,ds\\
    ={}& \phi\int_0^t\Phi(s)\,dX_{\sigma}(s).
\end{align*}
So the theorem holds.
\end{proof}

All of the above properties are quite straightforward and generalize
the results of \citet{BNBPV:2012}. In the white noise analysis
setting, the \( \sS \)-transform (see \cref{def:3}) plays a central role and we next discuss
it in the context of the integrals we defined in the following
subsection.

\subsection{The \texorpdfstring{\( \sS \)-transform}{S-transform}}
\label{sec:S-transform}
We can apply some of the well known facts about the \( \sS
\)-transform and use the properties of the operator \( \sK_g \) to
find the \( \sS \)-transform of the integral with respect to a \vmbv{}
process.  Below we present two formulas for the \( \sS \)-transform of
the integrals in the case with no volatility modulation and with
modulation introduced through Wick product. We give explicit formulas
depending on the \( \sS \)-transform of the integrand only.

\begin{theorem}
  \label{thm:7}
  If \( \Phi(s) \) is integrable with respect to \( dX_1(s) \) on the
  interval \( [0,t] \), then
  \begin{align}
    \begin{split}
      \sS \left( \int_0^t\Phi(s)\,dX_1(s) \right) ={}&
      \int_0^t \sK_g(\sS(\Phi)(\xi))(t,s) \xi(s)\,ds\\
      & + \int_0^t\frac{\delta}{\delta\xi(s)}\left\{
        \sK_g(\sS(\Phi)(\xi))(t,s) \right\}\,ds.\label{eq:17}
    \end{split}
  \end{align}
\end{theorem}
\begin{proof}
  It is easy to see that \(
  \sS(\sK_g(\Phi))(\xi)=\sK_g(\sS(\Phi)(\xi)) \) because the \( \sS
  \)-transform is linear and the Lebesgue measure in \cref{prop:7-1}
  can be substituted by any measure. Now, \cref{eq:17} is a simple
  consequence of \cref{prop:7-2,prop:7-3}.
\end{proof}

\begin{theorem}
  If \( \Phi(s) \) is integrable with respect to \(
  X_{\diamond\Sigma}(s) \) on the interval \( [0,t] \), then
  \begin{align}
    \begin{split}
      \label{eq:26}
      \sS \left( \int_0^t\Phi(s)\,dX_{\diamond\Sigma}(s) \right) ={}&
      \int_0^t \sK_g(\sS(\Phi)(\xi))(t,s)
      \cdot \sS(\Sigma(s))(\xi)  \xi(s)\,ds\\
      & + \int_0^t\frac{\delta}{\delta\xi(s)}\left\{ \sK_g(
        \sS(\Phi)(\xi))(t,s) \cdot\sS(\Sigma(s))(\xi) \right\}\,ds.
    \end{split}
  \end{align}
\end{theorem}
\begin{proof}
  This follows from a reasoning similar to the one in the proof of
  \cref{thm:7} with the additional use of the fact that \(
  \sS(\Phi\diamond\Psi)=\sS(\Phi)\cdot\sS(\Psi) \).
\end{proof}
\begin{remark}
  Observe that \cref{eq:26} holds also in the case of strong
  independence discussed in \cref{cor:1}.
\end{remark}

\subsection{Chaos expansion}
\label{sec:chaos}
In this section, we give explicit chaos expansions in both cases, of no
volatility modulation and with the volatility introduced through the
Wick product. It is possible to find the chaos expansion for the \(
dX_{\sigma} \) integral, however the complexity of the formula renders
it almost useless.

\begin{theorem}
  \label{thm:8}
  If \( \Phi(s) \) is \( dX_1(s) \)-integrable on the interval \(
  [0,t] \), then
  \begin{align*}
    \int_0^t\Phi(s)\,dX_1(s) ={}& \int_0^t\sK_g(\Phi^{(1)})(t,s)\,ds\\
    &+ \sum_{n=1}^{\infty} I_n \left( \sK_g(\Phi^{(n-1)})(t,\cdot) +
      (n+1)\int_0^t\sK_g(\tilde{\Phi}^{(n+1)})(t,s)\,ds \right),
  \end{align*}
  where \( \tilde{\Phi}^{(n+1)}(x_1,\ldots,x_n,s) =
  \Phi^{(n+1)}(x_1,\ldots,x_n,s,s) \).
\end{theorem}
\begin{proof}
  Suppose that \( \Phi(t)=\sum_{n=0}^{\infty}I_n(\Phi^{(n)}(t)) \). It
  is not difficult to see that with the application of the stochastic
  Fubini theorem we have
  \begin{equation*}
    \sK_g(\Phi)(t,s) = \sum_{n=0}^{\infty}I_n(\sK_g(\Phi^{(n)})(t,s)).
  \end{equation*}
  Hence we have the following
  \begin{align}
    \int_0^t \sK_g(\Phi)(t,s)\,\delta B(s) ={}&
    \sum_{n=0}^{\infty}I_{n+1}(\sK_g(\Phi^{(n)})(t,\cdot))\nonumber\\
    ={}&\sum_{n=1}^{\infty}I_{n}(\sK_g(\Phi^{(n-1)})(t,\cdot)).
    \label{eq:18}
  \end{align}
  Also,
  \begin{align*}
    D_s\sK_g(\Phi)(t,s) ={}&
    \sum_{n=0}^{\infty}nI_{n-1}(\sK_g(\tilde{\Phi}^{(n)})(t,s))\\
    ={}& \sK_g(\tilde{\Phi}^{(1)})(t,s)+
    \sum_{n=1}^{\infty}(n+1)I_n(\sK_g(\tilde{\Phi}^{(n+1)})(t,s)),
  \end{align*}
  where \( \tilde{\Phi}^{(n)}(x_1,\ldots,x_{n-1},s) =
  \Phi^{(n+1)}(x_1,\ldots,x_{n-1},s,s) \), because the stochastic
  derivative is taken in \( s \) and \( \Phi \) already depends on \(
  s \).  Hence,
  \begin{align}
    \begin{split}
      \label{eq:19}
      \int_0^tD_s\sK_g(\Phi)(t,s)\,ds ={}&
      \int_0^t \sK_g(\tilde{\Phi}^{(1)})(t,s)\,ds\\
      &+ \sum_{n=1}^{\infty}I_n\left((n+1)\int_0^t
        \sK_g(\tilde{\Phi}^{(n+1)})(t,s)\,ds\right),
    \end{split}
  \end{align}
  where we have again used the stochastic Fubini theorem.  Putting
  \cref{eq:18,eq:19} together, we obtain the desired result.
\end{proof}

\begin{theorem}
  If \( \Phi(s) \) is \( dX_{\diamond\Sigma}(s) \)-integrable on the interval
  \( [0,t] \), then
  \begin{align*}
    \int_0^t\Phi(s)\,dX_{\diamond\Sigma}(s) ={}&
    \int_0^t\sK_g(\Phi^{(1)})(t,s)\widehat{\otimes}\Sigma^{(0)}(s)\,ds\\
    &+\sum_{n=0}^{\infty} I_n \Bigg(
    \sum_{m=0}^{n-1}\sK_g(\Phi^{(n-1-m)})(t,s)
    \widehat{\otimes}\Sigma^{(m)}(s) \\
    &\qquad+(n+1)\sum_{m=0}^n\int_0^t\sK_g(\Phi^{(n+1-m)})(t,s)
    \widehat{\otimes}\Sigma^{(m)}(s)\,ds \Bigg).
  \end{align*}
\end{theorem}
\begin{proof}
  We can establish the formula above using the same arguments as in
  the proof of \cref{thm:8} with the addition of the formula for the
  Wick product given in \cref{eq:20}.
\end{proof}
\begin{remark}
  The above holds in the case of strong independence discussed in
  \cref{cor:1}.
\end{remark}

\subsection{Stability}
In this section, we show that strong convergence of \( \Phi_n \) to
\(\Phi\) implies strong convergence of \( \int_0^t\Phi_n(s)\,dX_1(s)\)
to \(\int_0^t\Phi(s)\,dX_1(s) \).

\begin{theorem}
  Suppose that \( \Phi_n,\Phi \) are \( dX_1 \)-integrable. Suppose
  also that for some \( \lambda>0 \) and almost all \( t\in[0,T] \) we
  have \( \norm{\Phi_n(t)-\Phi(t)}_{-\lambda}\to 0 \) and \(
  \norm{\Phi_n(t)-\Phi(t)}_{-\lambda}\leq h(t) \), where \( h\in
  L^1([0,T]) \) .  Then for any \( \varepsilon>0 \)
  \begin{equation*}
    \lim_{n\to\infty}
    \norm*{\int_0^t\Phi_n(s)\,dX_1(s)-\int_0^t\Phi(s)\,dX_1(s)}_{-\lambda-\varepsilon}
    =0.
  \end{equation*}
\end{theorem}
\begin{proof}
  By linearity of the integral and the triangle inequality we have
  \begin{align*}
    \MoveEqLeft\norm*{\int_0^t\Phi_n(s)\,dX_1(s)-\int_0^t\Phi(s)\,dX_1(s)}_{-\lambda-\varepsilon}\\
    ={}& \norm*{\int_0^t\Phi_n(s)-\Phi(s)\,dX_1(s)}_{-\lambda-\varepsilon}\\
    \leq{}& \norm*{\int_0^t\sK_g(\Phi_n-\Phi)(t,s)\,\delta B(s)
    }_{-\lambda-\varepsilon}
    +\norm*{\int_0^tD_s\sK_g(\Phi_n-\Phi)(t,s)\,ds}_{-\lambda-\varepsilon}
  \end{align*}
  It is enough to show that both of the norms above converge to zero
  as \( n\to\infty \). First, we estimate the square of the norm of \(
  \sK_g(\Phi_n-\Phi) \) as \( n\to 0 \) as it will be useful
  later. Below, we use a part of the proof of \cref{prop:1}, where we
  have shown that
  \begin{equation*}
    \norm*{\int_s^t\Phi(u)-\Phi(s)\,g(du,s)}^2_{-\lambda} \leq
    V_s^t[g(\cdot,s)]\int_s^t\norm{\Phi(u)-\Phi(s)}^2_{-\lambda}\,\abs*{g}(du,s).
  \end{equation*}
  Now, consider
  \begin{align*}
    \MoveEqLeft\norm*{\sK_g(\Phi_n-\Phi)(t,s)}_{-\lambda}^2\\
    ={}&\norm*{g(t,s)(\Phi_n(s)-\Phi(s))+\int_s^t\big[(\Phi_n(u)-\Phi(u))-(\Phi_n(s)-\Phi(s))\big]\,g(du,s)}^2_{-\lambda}\\
    \leq{}& 2\abs*{g(t,s)}^2\norm*{\Phi_n(s)-\Phi(s)}^2_{-\lambda}\\
    &+2\norm*{\int_s^t\big[(\Phi_n(u)-\Phi(u))-(\Phi_n(s)-\Phi(s))\big]\,g(du,s)}^2_{-\lambda}\\
    \leq{}& 2\abs*{g(t,s)}^2\norm*{\Phi_n(s)-\Phi(s)}^2_{-\lambda}\\
    &+2V_s^t[g(\cdot,s)]\int_s^t\norm*{(\Phi_n(u)-\Phi(u))-(\Phi_n(s)-\Phi(s))}^2_{-\lambda}\,\abs*{g}(du,s)\\
    \leq{}& 2\abs*{g(t,s)}^2\norm*{\Phi_n(s)-\Phi(s)}^2_{-\lambda}\\
    &+4V_s^t[g(\cdot,s)]\int_s^t\left[\norm*{\Phi_n(u)-\Phi(u)}^2_{-\lambda}+\norm{\Phi_n(s)-\Phi(s)}^2_{-\lambda}\right]\,\abs*{g}(du,s)\\
    \to{}& 0,\qquad\text{as }n\to\infty,
  \end{align*}
  by Lebesgue's dominated convergence theorem because \(
  \norm*{\Phi_n(t)-\Phi(t)}_{-\lambda}\to 0 \) for almost all \( t \).

  Now, by \cref{thm:2}, there is a constant \( C_{\varepsilon} \) such that
  \begin{align*} \norm*{\int_0^t\sK_g(\Phi_n-\Phi)(t,s)\,\delta B(s)
    }_{-\lambda-\varepsilon} \leq{}& C_{\varepsilon}\int_0^t
    \norm*{\sK_g(\Phi_n-\Phi)(t,s)}^2_{-\lambda}\,dt\\ \to{}& 0
  \end{align*} as \( n\to\infty \).

  Finally, by H\"older's inequality and \cref{thm:1}, we have
  \begin{align*}
    \norm*{\int_0^tD_s(\sK_g(\Phi_n-\Phi)(t,s))\,ds}_{-\lambda-\varepsilon}\leq{}&
    \int_0^t\norm*{D_s(\sK_g(\Phi_n-\Phi)(t,s))}_{-\lambda-\varepsilon}\,ds\\
    \leq{}&t\int_0^t\norm*{D_s(\sK_g(\Phi_n-\Phi)(t,s))}^2_{-\lambda-\varepsilon}\,ds\\
    \leq{}&t\tilde{C}_{\varepsilon}\norm*{\sK_g(\Phi_n-\Phi)(t,s)}^2_{-\lambda}\\ \to{}&0
  \end{align*}
  as \( n\to \infty \). Thus the result holds.
\end{proof}

The following two theorems restate the above result in the setting
with smooth volatility and the volatility introduced through the Wick
product. We omit the proofs as they follow the same argument as the
proof of the results above with additional use of some of the norm
estimates from previous section.

\begin{theorem}
  Suppose that \( \Phi_n,\Phi \) are \( dX_\sigma \)-integrable.
  Suppose also that for some \(
  \lambda>\nu=\tfrac{1}{2}\ln(2+\sqrt{2}) \) and almost all \(
  t\in[0,T] \) we have \( \norm{\Phi_n(t)-\Phi(t)}_{-\lambda+\nu}\to 0
  \) and \( \norm{\Phi_n(t)-\Phi(t)}_{-\lambda+\nu}\leq h(t) \), where
  \( h\in L^1([0,T]) \) .  Then for any \( \varepsilon>0 \)
  \begin{equation*}
    \lim_{n\to\infty}
    \norm*{\int_0^t\Phi_n(s)\,dX_\sigma(s)-\int_0^t\Phi(s)\,dX_\sigma(s)}_{-\lambda-\varepsilon}
    =0.
  \end{equation*}
\end{theorem}

\begin{theorem}
  Suppose that \( \Phi_n,\Phi \) are \( dX_{\diamond\Sigma}
  \)-integrable.  Suppose also that for some \(
  \lambda>\nu=\tfrac{1}{2}\ln(2+\sqrt{2}) \) and almost all \(
  t\in[0,T] \) we have \( \norm{\Phi_n(t)-\Phi(t)}_{-\lambda+\nu}\to 0
  \) and \( \norm{\Phi_n(t)-\Phi(t)}_{-\lambda+\nu}\leq h(t) \), where
  \( h\in L^1([0,T]) \) .  Then for any \( \varepsilon>0 \)
  \begin{equation*}
    \lim_{n\to\infty}
    \norm*{\int_0^t\Phi_n(s)\,dX_{\diamond\Sigma}(s)-\int_0^t\Phi(s)\,dX_{\diamond\Sigma}(s)}_{-\lambda-\frac{1}{2}-\varepsilon}
    =0.
  \end{equation*}
\end{theorem}

\section{An example -- the Donsker delta function}
\label{sec:example}
In this section, we present an example of a generalized process that
cannot be integrated in the setting of \citet{BNBPV:2012}. We study
the integrability of the Donsker delta function with respect to a
Brownian-driven Volterra process built upon an Ornstein--Uhlenbeck kernel
function \( g \). The importance of the Donsker delta function is
well-illustrated in \citet{AOU:2001}, where the authors use the
Donsker delta function to compute hedging strategies.

It is well-known that the Donsker delta function, that is \(
\delta_0(B(t)) \), is not an \( (L^2) \) stochastic process, however
it is a process in \( \sG_{-\lambda} \) for any \( \lambda>0 \)
\citep[see][Example 2.2]{PT:1995} and it has a chaos expansion given
by
\begin{equation*}
  \delta_0(B(t))
  = \frac{1}{\sqrt{4\pi t}}\sum_{n=0}^{\infty}
  I_{2n}\left( \frac{(-1)^n}{(2t)^nn!}\mathbbm{1}_{[0,t)}^{\otimes 2n} \right).
\end{equation*}
We also have the following formula for the norm of \( \delta_0(B(t))
\)
\begin{equation*}
  \norm*{\delta_0(B(t))}_{-\lambda}^2
  =\frac{1}{2\pi t}\sum_{n=0}^{\infty}\frac{(2n)!}{4^n(n!)^2e^{\lambda4n}},
\end{equation*}
where the sum converges for any \( \lambda>0 \). Therefore
\begin{equation*}
  \norm*{\delta_0(B(t))}_{-\lambda}^2=\frac{1}{t}C_{\lambda},
\end{equation*}
where \( C_{\lambda} \) is a constant depending only on \( \lambda \).

Now, we take \( g(t,s)=e^{-\alpha(t-s)} \), and will show that for any
\( \varepsilon>0 \), \(
\Phi(t)=\mathbbm{1}_{[\varepsilon,\infty)}(t)\delta_0(B(t)) \) is \(
dX_1(t) \)-integrable. We need to avoid \( 0 \) as \( \delta_0(B(0))
\) does not exist. It can be easily verified that
\cref{as:A} \cref{item:1,item:2} are satisfied with our choice of \( \Phi \)
and \( g \). In order to show that \cref{as:A} \cref{item:3} holds, take \(
\varepsilon< s<t<T \) for some \( 0<\varepsilon<T \) and consider
\begin{align}
  \int_s^t\norm*{\Phi(u) - \Phi(s)}_{-\lambda}^2 \,\abs*{g}(du,s) \leq{}&
  2\int_s^t\left(\norm*{\Phi(u)}_{-\lambda}^2
    + \norm*{\Phi(s)}_{-\lambda}^2\right) \,\abs*{g}(du,s)\nonumber\\
  \leq{}&
  4\int_s^t \norm*{\Phi(s)}_{-\lambda}^2\,\abs*{g}(du,s)\label{eq:31}\\
  ={}& 4 C_{\lambda}\frac{1}{s}
  \left(1-  e^{-\alpha (t-s)}\right)\nonumber\\
  <{}&\infty,\nonumber
\end{align}
where we have used monotonicity of function \( \tfrac{1}{t} \).

To show that \cref{as:B} holds, one can apply arguments similar to the
ones used above together with some of the properties of the
exponential integral function \( E_i=\int_{-\infty}^x\tfrac{e^t}{t}\,
dt \). We omit these computations because they are straightforward but
rather tedious.

So \( \mathbbm{1}_{[\varepsilon,\infty)}(t)\delta_0(B(t)) \) is \(
dX_1(t) \)-integrable. Moreover, since \(
\mathbbm{1}_{[\varepsilon,\infty)}(t)\delta_0(B(s))\in\sG_{-\lambda}
\) for any \( \lambda>0 \), by \cref{thm:4}, we have that
\begin{equation}
\label{eq:30}
  \int_0^t\mathbbm{1}_{[\varepsilon,\infty)}(t)\delta_0(B(s))\,dX_1(s) \in \sG_{-\lambda}
\end{equation}
for any \( \lambda>0 \).

Since the chaos expansion of the integral in \cref{eq:30} is rather
long and complex, we give the chaos expansion of \(
\sK_g(\delta_0(B))(t,s) \) to show an intermediate step in the
derivation of the complete chaos expansion of the integral.  And so
\begin{align*}
  \sK_g(\Phi)(t,s)
  ={}& \frac{1}{4\pi} \sum_{n=0}^{\infty} I_{2n}
  \Bigg(
    \frac{\alpha e^{-\alpha(t-s)}+e^{-\alpha(t-s)}-1}{\alpha s^{n+1}}
    \mathbbm{1}_{[0,s)}^{\otimes 2n}(v_1\ldots v_{2n})\\
    &\qquad+ e^{\alpha s}\alpha^n
    \left(
      \Gamma(-n,\alpha s)
      - \Gamma(-n,\alpha\cdot\min\{t,v_1,\ldots v_{2n}\})
    \right)
  \Bigg),
\end{align*}
where \( \Gamma(\nu,x)=\int_x^{\infty}t^{\nu-1}e^{-t}\,dt \) is the
incomplete Gamma function.

Note that with a different Volterra kernel, it may be possible to
balance the infinite norm of the Donsker delta at zero without
resorting to an explicit cut-off function like \(
\mathbbm{1}_{[\varepsilon,\infty)}(t) \). Looking at \cref{eq:31}, an
obvious and rather trivial example is a Volterra kernel of the form \(
g(t,s)=a(t)b(s) \), where \( b(s) \) and \( a(s)b(s) \) are functions
that are decaying to \( 0 \) at an at most linear rate as \( s\to
0^{+} \). Also from \cref{eq:31} we see that it is impossible to find
a shift-kernel such that \( \delta_0(B(s)) \) is \( dX_1(s)
\)-integrable on \( [0,\varepsilon] \) for any \( \varepsilon>0 \).

\section{Conclusions}
We have extended the theory of integration with respect to volatility
modulated Brownian-driven Volterra processes first discussed in
\citet{BNBPV:2012} onto the space of generalized Potthoff--Timpel
distributions \( \sG^{*} \). We have employed the white noise analysis
tools to show the properties of the stochastic derivative and Skorohod
integral in the space \( \sG^{*} \) as well as numerous properties of
the \vmbv{} integral without volatility modulation and with modulation
introduced in two different ways, through the pointwise product and through the Wick product.
We also show that under strong independence, the two volatility modulation approaches are equivalent.
Our approach allows to integrate,
for example, the Donsker delta function which is not an element of \(
(L^2) \) and thus not tractable in the setting of \citet{BNBPV:2012}.
Moreover, the theory presented in this paper allows for integration with respect to non-semimartingales (e.g. fractional Brownian motion) and for integration of non-adapted stochastic processes.

There are still some questions that are left without an answer. For
instance, it is natural to ask whether this approach can be
generalized to the setting of Hida spaces \( (\sS) \) and \( (\sS)^{*}
\). Another such question is that of the change of the driving
process. In \citet{BNBPV:2012}, the authors discuss not only Brownian
motion as the driver of the Volterra process, but also a pure-jump
L\'evy process, thus it is interesting to see whether our approach can
be applied in that setting.

Another possible generalization opportunity comes from the fact that
for now, the domain of integration is a finite interval, namely \(
[0,T] \).  Recently \citet{BOCGP:2010} studied integration theory on
the real line that may be used to define integrals with respect to
processes of the form
\begin{equation}
\label{eq:22}
  X(t)=\int_{-\infty}^tg(t,s)\sigma(s)\,dL(s).
\end{equation}
Such processes are interesting both from theoretical and practical
perspective and it would be useful to extend the theory discussed in
the present paper in the setting of \cref{eq:22}.

\section*{Acknowledgments}
Ole E. Barndorff-Nielsen and Benedykt Szozda acknowledge support from
The T.N. Thiele Centre For Applied Mathematics In Natural Science and
from CREATES (DNRF78), funded by the Danish National Research
Foundation. Fred Espen Benth acknowledges financial support from the
two projects "Energy Markets: Modelling, Optimization and Simulation
(EMMOS)" and "Managing Weather Risk in Electricity Markets (MAWREM)",
funded by the Norwegian Research Council.

\bibliographystyle{plainnat}
\bibliography{wn-vmbv}{}

\end{document}